\documentclass[11pt]{amsart} % use larger type; default would be 10pt

\usepackage[utf8]{inputenc} % set input encoding (not needed with XeLaTeX)
\usepackage{xcolor}

\usepackage{geometry} % to change the page dimensions
\usepackage{color,graphicx} % support the \includegraphics command and options

\usepackage{amsmath}
\usepackage{amssymb}
\usepackage{amsthm}
\usepackage{chemarr}
\usepackage{esint}
\usepackage{cite}
\usepackage{booktabs} % for much better looking tables
\usepackage{array} % for better arrays (eg matrices) in maths
\usepackage{paralist} % very flexible & customisable lists (eg. enumerate/itemize, etc.)
\usepackage{mathrsfs}
\usepackage{verbatim} % adds environment for commenting out blocks of text & for better verbatim
\usepackage{dsfont}
\usepackage{mathtools}
\usepackage[yyyymmdd,hhmmss]{datetime}
\usepackage{subcaption}
\usepackage{fancyhdr} % This should be set AFTER setting up the page geometry
\newtheorem{thm}{Theorem}[section]
\newtheorem{defi}[thm]{Definition}
\newtheorem{prop}[thm]{Proposition}
\newtheorem{lma}[thm]{Lemma}
\newtheorem{rem}[thm]{Remark}
\newtheorem{assump}[thm]{Assumption}

\newtheorem{notation}[thm]{Notation}
\DeclareMathOperator{\diam}{diam}

\DeclareMathOperator{\size}{size}

\DeclareMathOperator{\dist}{dist}

\DeclareMathOperator{\conv}{conv}

\DeclareMathOperator{\diag}{diag}

\newcommand{\id}{\mathrm{id}}

\newcommand{\e}{\eps}
\newcommand{\weakstar}{\stackrel{\ast}{\rightharpoonup}}
\newcommand{\wsto}{\weakstar}

\renewcommand{\H}{{\mathscr{H}}}
\renewcommand{\d}{\mathrm{d}}

\newcommand{\Gr}{\mathrm{Gr}\,}
\renewcommand{\L}{\mathscr{L}}
\def\N{{\mathbb N}} %naturals
\def\Z{{\mathbb Z}} %integers
\def\R{{\mathbb R}} %reals

\def\Hm{{\mathscr H}}
\def\cT {{\mathcal T}}
\def\M{{\mathcal M}}
\def\B{{\mathcal B}}
\def\l#1#2{l_{#1#2}}

\newcommand{\dd}{\; \mathrm{d}}
\newcommand{\eps}{\varepsilon}

\def\XXint#1#2#3{{\setbox0=\hbox{$#1{#2#3}{\int}$ }
\vcenter{\hbox{$#2#3$ }}\kern-.6\wd0}}

\begin{document}

\title{Approximation of the Willmore energy by a discrete geometry model}
\date{\today} 
\author[P.~Gladbach]{Peter Gladbach}
\author[H.~Olbermann]{Heiner Olbermann} 
\address[Peter Gladbach]{Institut f\"ur Angewandte Mathematik, Universit\"at Bonn, 53115 Bonn, Germany}
\address[Heiner Olbermann]{Institut de Recherche en Math\'ematique et Physique, UCLouvain, 1348 Louvain-la-Neuve, Belgium}
\email[Peter Gladbach]{gladbach@iam.uni-bonn.de}
\email[Heiner Olbermann]{heiner.olbermann@uclouvain.be}

\begin{abstract}
  We prove that a certain discrete energy for triangulated  surfaces, defined in the spirit of discrete differential geometry, converges to the Willmore energy in the sense of $\Gamma$-convergence. Variants of this discrete energy have been discussed before in the computer graphics literature. 
\end{abstract}

\maketitle

\section{Introduction}

The numerical analysis of elastic shells  is a
vast field with important applications in physics and engineering. In most
cases, it is carried out via the finite element method. In the physics and computer
graphics literature, there have been suggestions to use simpler methods  based
on discrete differential geometry \cite{meyer2003discrete,bobenko2008discrete}.  Discrete differential geometry of surfaces is the study of triangulated polyhedral surfaces. (The epithet ``simpler'' has to be understood as ``easier to implement''.) We mention in passing that models based on triangulated polyhedral surfaces have applications in materials science beyond the elasticity of thin shells. E.g., recently these models have been used to describe  defects in nematic liquids on thin shells \cite{canevari2018defects}. This amounts to a generalization to arbitrary surfaces of the discrete-to-continuum analysis for the XY model in two dimensions that leads to  Ginzburg-Landau type models in the continuum limit \cite{MR2505362,alicandro2014metastability}.

\medskip

Let us describe some of the methods mentioned above in  more detail.
Firstly, there are the so-called
\emph{polyhedral membrane models} which in fact can be used for a whole
array of physical and engineering problems (see e.g.~the review
\cite{davini1998relaxed}). In the context of plates and shells,  the
so-called Seung-Nelson model \cite{PhysRevA.38.1005} is widely used.
This associates membrane and bending energy to a piecewise affine map $y:\R^2\supset U\to \R^3$, where the pieces are determined by a triangulation $\mathcal T$ of the polyhedral domain $U$. The bending energy is given by
  \begin{equation}
E^{\mathrm{SN}}(y)=  \sum_{K,L} |n(K)-n(L)|^2\,,\label{eq:1}
  \end{equation}
where the sum runs over those unordered pairs of triangles $K,L$ in $\mathcal T$  that share an edge, and $n(K)$ is the surface normal on the triangle $K$. In \cite{PhysRevA.38.1005}, it has been argued that for a fixed
limit deformation $y$, the energy \eqref{eq:1} should approximate the Willmore energy
  \begin{equation}
  E^{\mathrm{W}}(y)=\int_{y(U)} |Dn|^2\dd\Hm^2\label{eq:2}
  \end{equation}
when the grid size of the triangulation $\mathcal T$ is sent to 0, and the argument of the discrete energy \eqref{eq:1} approximates the (smooth) map $y$. In \eqref{eq:2} above, $n$ denotes the surface normal and $\H^2$ the two-dimensional Hausdorff measure.
These statements have been made more precise in \cite{schmidt2012universal}, where it has been shown that the result of the limiting process depends on the  used triangulations. In particular, the following has been shown in this reference: For $j\in\N$, let $\mathcal T_j$ be a triangulation of $U$ consisting of equilateral triangles such that one of the sides of each triangle is parallel to the $x_1$-direction, and such that the triangle size tends 0 as $j\to\infty$. Then the limit energy reads
\[
  \begin{split}
  E^{\mathrm{FS}}(y)=\frac{2}{\sqrt{3}}\int_U &\big(g_{11}(h_{11}^2+2h_{12}^2-2h_{11}h_{22}+3h_{22}^2)\\
  &-8g_{12}h_{11}h_{12}+2 g_{22}(h_{11}^2+3h_{12}^2)\big)(\det g_{ij})^{-1}\dd x\,,
\end{split}
\]
where
\[
  \begin{split}
  g_{ij}&=\partial_i y\cdot\partial_j y\\
  h_{ij}&=n\cdot \partial_{ij} y \,.
\end{split}
\]
More precisely, if $y\in C^2(U)$ is given, then the sequence of maps $y_j$ obtained by piecewise affine interpolation of the values of $y$ on the vertices of the triangulations $\mathcal T_j$ satisfies
\[
  \lim_{j\to \infty}E^{\mathrm{SN}}(y_j)=E^{\mathrm{FS}}(y)\,.
  \]

Secondly, there is the more recent approach to using discrete differential
geometry for shells pioneered by Grinspun et al.~\cite{grinspun2003discrete}.
Their energy does not depend on an immersion $y$ as above, but is defined directly on  triangulated surfaces. Given such a surface   $\mathcal T$, the energy is given by

  \begin{equation}
  E^{\mathrm{GHDS}}(\mathcal T)=\sum_{K,L}  \frac{l_{KL}}{d_{KL}} \alpha_{KL}^2\label{eq:3}
  \end{equation}

  where the sum runs over unordered pairs of neighboring triangles $K,L\in\mathcal T$, $l_{KL}$ is the length of the interface between  $K,L$, $d_{KL}$ is the distance between the centers of the circumcircles of $K,L$, and $\alpha_{KL}$ is the difference of the angle between $K,L$ and $\pi$, or alternatively the angle between the like-oriented normals $n(K)$ and $n(L)$, i.e. the \emph{dihedral angle}.

In \cite{bobenko2005conformal}, Bobenko has defined an energy for piecewise affine surfaces $\mathcal T$ that is invariant under conformal transformations. It is defined via the circumcircles of triangles in $\mathcal T$, and the external intersection angles of circumcircles of neighboring triangles. Denoting this intersection angle for neighboring triangles $K,L$ by $\beta_{KL}$, the energy reads
  \begin{equation}\label{eq:4}
  E^\mathrm{B} (\mathcal T) = \sum_{K,L}\beta_{KL}-\pi\, \#\,\text{Vertices}(\mathcal T)\,.
\end{equation}
Here $\text{Vertices}(\mathcal T)$ denotes the vertices of the triangulation $\mathcal T$, the sum is again over nearest neighbors.
It has been shown in \cite{bobenko2008surfaces} that this energy is the same as
\eqref{eq:3} up to terms that vanish as the size of triangles is sent to zero
(assuming sufficient smoothness of the limiting surface). The reference
\cite{bobenko2008surfaces} also contains an analysis of the energy for this
limit. If the limit surface is smooth, and it is approximated by triangulated
 surfaces $\mathcal T_\e$ with maximal triangle size $\e$  that satisfy
a number of technical assumptions, then the Willmore energy of the limit surface
is smaller than or equal to the limit of the energies \eqref{eq:3} for the approximating surfaces, see Theorem 2.12 in \cite{bobenko2008surfaces}. The technical assumptions are
\begin{itemize}
\item each vertex in the triangulation $\mathcal T_\e$ is connected to six other vertices by edges,
  \item the lengths of the sides of the hexagon formed by six triangles that share one vertex differ by at most $O(\e^4)$,
\item neighboring triangles are congruent up to $O(\e^3)$.
  \end{itemize}
  Furthermore, it is stated that the limit is achieved if additionally the triangulation approximates a ``curvature line net''.

  \medskip

  The purpose of this present paper is to  generalize this convergence result, and put it into the framework of $\Gamma$-convergence \cite{MR1968440,MR1201152}. Instead of fixing the vertices of the polyhedral surfaces to lie on the limiting surfaces, we are going to assume that the convergence is weakly * in $W^{1,\infty}$ as graphs. This approach allows to completely drop the assumptions on the connectivity of vertices in the triangulations, and the assumptions of congruence -- we only need to require a certain type of regularity of the triangulations that prevents the formation of small angles. %

\medskip
  
  We are going to work with the energy
    \begin{equation}\label{eq:5}
    E(\mathcal T)=\sum_{K,L} \frac{l_{KL}}{d_{KL}} |n(K)-n(L)|^2\,,
  \end{equation}

  which in a certain sense  is equivalent to \eqref{eq:3} and \eqref{eq:4} in the limit of vanishing triangle size, see  the arguments from \cite{bobenko2008surfaces} and Remark \ref{rem:main} (ii) below.
  
 \medskip

 To put this approach into its context in the mathematical literature, we point out that it is another instance of a discrete-to-continuum limit, which has been a popular topic in mathematical analysis over the last few decades. We mention the seminal papers \cite{MR1933632,alicandro2004general} and the fact that  a variety of physical settings have been approached in this vein, such as spin and lattice systems \cite{MR1900933,MR2505362},  bulk elasticity \cite{MR2796134,MR3180690}, thin films \cite{MR2429532,MR2434899}, magnetism \cite{MR2186037,MR2505364}, and many more.
 
\medskip

  The topology that we are going to use in our $\Gamma$-convergence statement is much coarser  than the one that corresponds to Bobenko's convergence result; however it is not the ``natural'' one that would yield compactness from finiteness of the energy \eqref{eq:5} alone. For a discussion of why we do not choose the latter see Remark \ref{rem:main} (i) below. Our topology is instead defined as follows:

Let $M$ be some fixed compact oriented two-dimensional $C^\infty$ submanifold of
$\R^3$ with normal $n_M:M\to S^2$.  Let $h_j\in W^{1,\infty}(M)$, $j=1,2,\dots$,
such that $\|h_j\|_{W^{1,\infty}}<C$ and $\|h_j\|_{\infty}<\delta(M)/2$ (where $\delta(M)$ is the \emph{radius of injectivity} of $M$, see Definition \ref{def:radius_injectivity} below) such that $
T_j:= \{x+h_j(x)n_M(x):x\in M\}$ are triangulated  surfaces (see Definition \ref{def:triangular_surface} below). We say
$\mathcal T_j\to \mathcal S:=\{x+h(x)n_M(x):x\in M\}$ if $h_j\to h$ in
$W^{1,p}(M)$ for all $1\leq p<\infty$. Our main theorem, Theorem \ref{thm:main}
below, is a $\Gamma$-convergence result in this topology.
The regularity assumptions that we impose on the triangulated surfaces under
considerations are  ``$\zeta$-regularity'' and the ``Delaunay property''. The
definition of these concepts can be found in Definition
\ref{def:triangular_surface} below.

  \begin{thm}
\label{thm:main}    
\begin{itemize}
  \item[(o)] Compactness: Let $\zeta>0$, and let $h_j$ be a bounded sequence in
    $W^{1,\infty}(M)$ such that $\mathcal T_j=\{x+h_j(x)n_M(x):x\in M\}$ is a
    $\zeta$-regular triangulated  surface and $\|h_j\|_\infty\leq\delta(M)/2$ for $j\in \N$ with $\limsup_{j\to\infty}E(\mathcal T_j)<\infty$. Then there exists a subsequence $h_{j_k}$ and $h\in W^{2,2}(M)$ such that $h_{j_k}\to h $ in $W^{1,p}(M)$ for every $1\leq p < \infty$.

  \item[(i)] Lower bound: Let $\zeta>0$. Assume that for $j\in\N$, $h_j\in W^{1,\infty}(M)$ with $\|h_j\|\leq \delta(M)/2$,  $\mathcal T_j:=\{x+h_j(x)n_M(x):x\in M\}$ is a 
    $\zeta$-regular triangulated  surface fulfilling the Delaunay
    property, and that $\mathcal T_j\to S=\{x+h(x)n_M(x):x\in M\}$ for $j\to\infty$. Then % $h\in W^{2,2}(M)$ and
         \[
        \liminf_{j\to\infty} E(\mathcal T_j)\geq \int_{S} |Dn_S|^2\dd\H^2\,.
      \]
    \item[(ii)] Upper bound: Let $h\in W^{1,\infty}(M)$ with $\|h\|_\infty\leq \delta(M)/2$ and
      $S=\{x+h(x)n_M(x):x\in M\}$. Then there exists $\zeta>0$ and a sequence $(h_j)_{j\in\N}\subset W^{1,\infty}(M)$ such that $\mathcal T_j:=\{(x+h_j(x)n_M(x):x\in M\}$ is a 
      $\zeta$-regular triangulated  surface satisfying the
      Delaunay property for each $j\in \N$, and we have $\mathcal T_j\to S$ for $j\to \infty$ and 
        \[
        \lim_{j\to\infty} E(\mathcal T_j)= \int_{S} |Dn_S|^2\dd\H^2\,.
      \]
    \end{itemize}
  \end{thm}

  \begin{rem}\label{rem:main}
    \begin{itemize}
\item[(i)] We are not able to derive a convergence result in a topology that
  yields convergence from boundedness of the energy \eqref{eq:5} alone. Such an
  approach would necessitate the interpretation of the surfaces as varifolds or
  currents. To the best of our knowledge, the theory of 
  integral functionals on varifolds (see e.g.~\cite{menne2014weakly,hutchinson1986second,MR1412686})  is not
  developed to the point to allow for a treatment of this question. In particular, there does not exist a sufficiently
  general theory of lower semicontinuity of integral functionals for varifold-function pairs.
\item[(ii)] We can state
  analogous results based on the energy functionals \eqref{eq:3},
  \eqref{eq:4}. To do so, our proofs only need to be modified slightly: As soon as we have reduced the situation to the graph case
  (which we do by assumption), the upper bound construction can be carried out
  as here; the smallness of the involved dihedral angles assures that the
  arguments from \cite{bobenko2005conformal} suffice to carry through the proof.
    Concerning the lower bound, we also reduce to the case of small dihedral angles by a blow-up procedure around Lebesgue points of the derivative of the surface normal of the limit surface. (Additionally, one can show smallness of the contribution of a few pairs of triangles whose dihedral angle is not small.) Again, the considerations from \cite{bobenko2005conformal} allow for a translation of our proof to the case of the  energy functionals \eqref{eq:3},
    \eqref{eq:4}.
  \item[(iii)] As we will show in Chapter \ref{sec:necess-dela-prop}, we need to require the Delaunay property in order to obtain the lower bound statement. Without this requirement, we will show that a hollow cylinder can be approximated by triangulated surfaces with arbitrarily low energy, see Proposition~\ref{prop: optimal grid}.
    \item[(iv)] Much more general approximations of surfaces by discrete geometrical objects have recently been proposed in \cite{buet2017varifold,buet2018discretization,buet2019weak}, based on tools from the theory of  varifolds.  
    \end{itemize}
  \end{rem}

  \subsection*{Plan of the paper}
  In Section \ref{sec:defin-prel}, we will fix definitions and make some preliminary observations on triangulated surfaces. The proofs of the compactness and lower bound part will be developed in parallel in Section \ref{sec:proof-comp-lower}. The upper bound construction is carried out in Section \ref{sec:surf-triang-upper}, and in Section \ref{sec:necess-dela-prop} we demonstrate that the requirement of the Delaunay property is necessary in order to obtain the lower bound statement.

\section{Definitions and preliminaries}
\label{sec:defin-prel}
\subsection{Some general notation}
  \begin{notation}
    For a two-dimensional submanifold $M\subset\R^3$, the tangent space of $M$ in $x\in M$ is
    denoted by $T_{x}M$. For functions $f:M\to\R$, we denote their gradient by $\nabla f\in T_xM$; the norm $|\cdot|$ on $T_xM\subset\R^3$ is the Euclidean norm inherited from $\R^3$. For $1\leq p\leq \infty$, we denote by $W^{1,p}(M)$ the space of functions $f\in L^p(M)$ such that $\nabla f\in L^p(M;\R^3)$, with norm
    \[
      \|h\|_{W^{1,p}(M)}=\|f\|_{L^p(M)}+\|\nabla f\|_{L^p(M)}\,.
      \]

      % For linear spaces $X,Y$, the space of linear operators
      % $X\to Y$ is denoted by $\mathcal L(X;Y)$, with $\mathcal L(X) = \mathcal L(X;X)$.
      For $U\subset\R^n$ and a function
    $f:U\to\R$, we denote the graph of $f$ by
    \[
      \Gr f=\{(x,f(x)):x\in U\}\subset\R^{n+1}\,.
    \]
    For $x_1,\dots,x_m\subset \R^k$, the convex hull of $\{x_1,\dots,x_m\}$ is
    denoted by
    \[
      [x_1,\dots,x_m]=\left\{\sum_{i=1}^m \lambda_ix_i:\lambda_i\in [0,1] \text{
          for } i=1,\dots,m, \, \sum_{i=1}^m\lambda_i=1\right\}\,.
    \]
    We will identify $\R^2$ with the subspace $\R^2\times\{0\}$ of $\R^3$. The $d-$dimensional Hausdorff measure is denoted by $\H^d$, the $k-$dimensional Lebesgue measure by $\L^k$.  The
    symbol ``$C$'' will be used as follows: A statement such as
    ``$f\leq C(\alpha)g$'' is shorthand for ``there exists a constant $C>0$ that
    only depends on $\alpha$ such that $f\leq Cg$''. The value of $C$ may change
    within the same line.  For $f\leq C g$, we also write
    $f\lesssim g$.
  \end{notation}

\subsection{Triangulated surfaces: Definitions}
\begin{defi}
\label{def:triangular_surface}
  \begin{itemize}
  \item [(i)] A \textbf{triangle} is the convex hull $[x,y,z]\subset \R^3$ of three points $x,y,z \in \R^3$. A \textbf{regular} triangle is one where $x,y,z$ are not colinear, or equivalently $\Hm^2([x,y,z])>0$.
  \item[(ii)] 
A \textbf{triangulated surface} is a finite collection
    $\cT = \{K_i\,:\,i = 1,\ldots, N\}$ of regular triangles
    $K_i = [x_i,y_i,z_i] \subset \R^3$ so that
    $\bigcup_{i=1}^N K_i \subset \R^3$ is a topological two-dimensional manifold
    with boundary; and the intersection of two different triangles $K,L\in \cT$ is either empty, a common vertex, or a common edge.
    
     We identify $\cT$ with its induced topological manifold
    $\bigcup_{i=1}^N K_i \subset \R^3$ whenever convenient. We say that $\cT$ is \textbf{flat} if
    there exists an affine subplane of $\R^3$ that contains $\cT$.

    % The set of {\bf edges} of $\cT_j$ is denoted by
    % \[
    %   \mathcal E(\cT_j)=\{\{K,L\}\subset\cT_j:
    %   K\cap L\neq\emptyset, K\neq L\}\,.
    %   \]

\item[(iii)]    The \textbf{size} of the triangulated surface, denoted $\size(\cT)$, is the
    maximum diameter of all its triangles.

   \item[(iv)] The triangulated surface $\cT$ is called $\zeta$\textbf{-regular}, with
    $\zeta > 0$, if the minimum angle in all triangles is at least $\zeta$ and
    $\min_{K\in \cT} \diam(K) \geq \zeta \size(\cT)$.
    % We say that $\cT$ is regular if it is $\zeta$-regular for some $\zeta>0$.

    \item[(v)] The triangulated surface satisfies the  \textbf{Delaunay}
      property if for every triangle
    $K = [x,y,z] \in \cT$ the following property holds: Let $B(q,r)\subset \R^3$ be the
    smallest ball such that $\{x,y,z\}\subset \partial{B(q,r)}$. Then $B(q,r)$ contains
    no vertex of any triangle in $\cT$. The point $q = q(K)\in \R^3$ is called
    the \textbf{circumcenter} of $K$, $\overline{B(q,r)}$ its
    \textbf{circumball} with circumradius $r(K)$, and $\partial B(q,r)$ its \textbf{circumsphere}.
  \end{itemize}

\end{defi}

 Note that triangulated surfaces have normals defined on all triangles and are compact
 and rectifiable.  For the argument of the circumcenter map
 $q$, we do not distinguish between triples of points $(a,b,c)\in \R^{3\times
   3}$ and the triangle $[a,b,c]$ (presuming $[a,b,c]$ is a regular triangle).

 \begin{notation}
   If $\cT=\{K_i:i=1,\dots,N\}$ is a triangulated surface, and
      $g:\cT\to \R$,
      then we identify $g$ with the function 
      $\cup_{i=1}^N K_i\to \R$ that is
      constant on the (relative) interior of each triangle $K$, and equal to
      $0$ on $K\cap L$ for $K\neq L\in \cT$. In particular we may write in this case
      $g(x)=g(K)$ for $x\in \mathrm{int}\, K$.
 \end{notation}

 \begin{defi}
   Let $\cT$ be a triangulated surface and  $K,L \in \cT$. We set
   \[
     \begin{split}
       \l{K}{L} &:= \H^1(K\cap L)\\
       d_{KL} &:= |q(K) - q(L)|
     \end{split}
   \]
   If $K,L$ are \textbf{adjacent}, i.e. if $\l{K}{L} > 0$, we may define $|n(K) - n(L)|\in \R$ as the norm of the difference of the normals $n(K),n(L)\in S^2$ which share an orientation, i.e. $2\sin \frac{\alpha_{KL}}{2}$, where $\alpha_{KL}$ is the dihedral angle between the triangles, see Figure \ref{fig:dihedral}. The discrete bending energy is then defined as
\[
E(\cT) = \sum_{K,L\in \cT} \frac{\l{K}{L}}{d_{KL}} |n(K) - n(L)|^2.
\]

Here, the sum runs over all unordered pairs of triangles. If $|n(K) - n(L)| = 0$ or $\l{K}{L} = 0$, the energy density is defined to be $0$ even if $d_{KL}=0$. If $|n(K) - n(L)| > 0$, $\l{K}{L} > 0$ and $d_{KL} = 0$, the energy is defined to be infinite.
\end{defi}

  \begin{figure}[h]
\begin{subfigure}{.45\textwidth}
  \begin{center}
    \includegraphics[height=5cm]{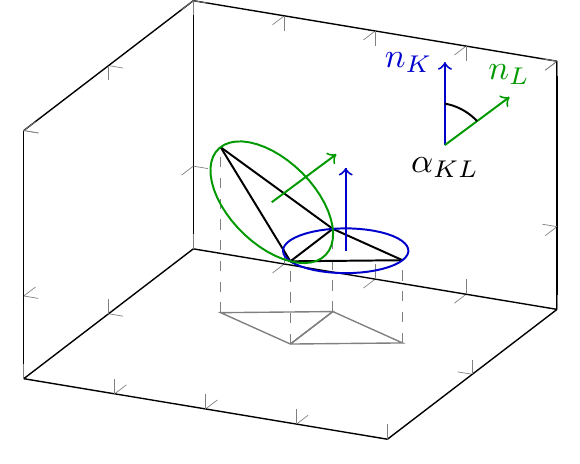}
  \end{center}
  \caption{ \label{fig:dihedral}}
\end{subfigure}
\hspace{5mm}
\begin{subfigure}{.45\textwidth}
\includegraphics[height=5cm]{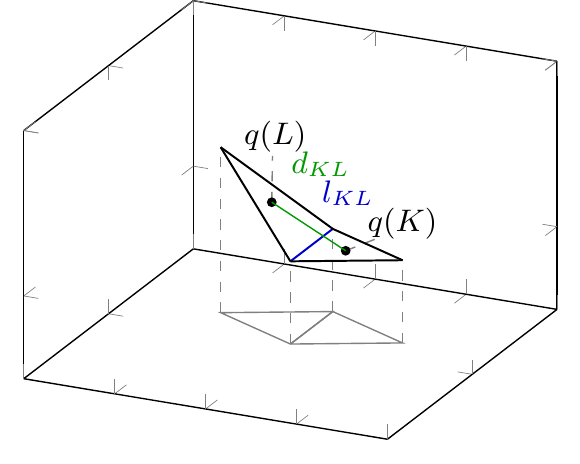}
\caption{\label{fig:dkllkl}} 
\end{subfigure}
\caption{($\mathrm{A}$) The dihedral angle $\alpha_{KL}$ for triangles $K,L$. It is related to the norm of the difference between the normals via $|n(K)-n(L)|=2\sin\frac{\alpha_{KL}}{2}$. ($\mathrm{B}$) Definitions of $d_{KL}$, $l_{KL}$.}
\end{figure}

\begin{notation}
  \label{not:thetaKL}
Let $H$ be an affine subplane of $\R^3$.
      For triangles $K,L\subset H$ that share an edge and $v\in\R^3$ parallel to
      $H$,  we define the function
      $\mathds{1}^v_{KL}:H \to \{0,1\}$ as $\mathds{1}_{KL}^v(x) = 1$ if and
      only if
      $[x,x+v]\cap (K\cap L) \neq \emptyset$. If the intersection $K\cap L$ does
      not consist of a single edge, then $\mathds{1}_{KL}^v\equiv
      0$. Furthermore, we let $\nu_{KL}\in \R^3$ denote the unit vector parallel to $H$
      orthogonal to the shared edge of $K,L$ pointing from $K$ to $L$ and
      \[
        \theta_{KL}^v=\frac{|\nu_{KL}\cdot v|}{|v|}\,.
        \]
    \end{notation}

See Figure \ref{fig:parallelogram} for an illustration of Notation \ref{not:thetaKL}.

\begin{figure}
\includegraphics[height=5cm]{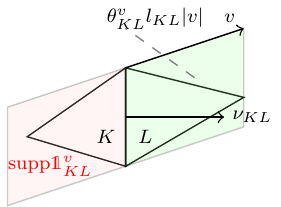}
\caption{Definition of $\theta_{KL}^v$: The parallelogram spanned by $v$ and the shared side $K\cap L$  has area $\theta^v_{KL}l_{KL}|v|$. This parallelogram translated by $-v$ is the support of $\mathds{1}_{KL}^v$. \label{fig:parallelogram}} 
\end{figure}

\medskip

We collect the notation that we have introduced for triangles and triangulated
surfaces for the reader's convenience in abbreviated form: Assume that  $K=[a,b,c]$ and $L=[b,c,d]$
are two regular triangles in $\R^3$. Then we have the following notation:

\begin{equation*}
\boxed{
  \begin{split}
    q(K)&: \text{ center of the smallest circumball for $K$}\\
    r(K)& :\text{ radius of the smallest circumball for $K$}\\
    d_{KL}&=|q(K)-q(L)|\\
    l_{KL}&:\text{ length of the shared edge of $K,L$}\\
 n(K)&: \text{ unit vector
      normal to $K$ }
  \end{split}
}
\end{equation*}

The following are defined if $K,L$
  are contained in an affine subspace $H$ of $\R^3$, and $v$ is a vector
  parallel to $H$:
\begin{equation*}
\boxed{
\begin{split}
\nu_{KL}&:\text{ unit vector  parallel to $H$
  orthogonal to}\\&\quad\text{ the shared edge of $K,L$ pointing from $K$ to $L$}\\ 
\theta_{KL}^v&=\frac{|\nu_{KL}\cdot v|}{|v|}\\
 \mathds{1}_{KL}^v&: \text{ function defined on $H$, with value one if}\\
 &\quad\text{  $[x,x+v]\cap (K\cap L)\neq \emptyset$, zero otherwise}
\end{split}
}
\end{equation*}

\subsection{Triangulated surfaces: Some preliminary observations}

For two adjacent triangles $K,L\in \cT$, we have $d_{KL} = 0$ if and only if the vertices of $K$ and $L$ have the same circumsphere. The following lemma states that for noncospherical configurations, $d_{KL}$ grows linearly with the distance between the circumsphere of $K$ and the opposite vertex in $L$.

\begin{lma}\label{lma: circumcenter regularity}
The circumcenter map $q:\R^{3\times 3} \to \R^3$ is $C^1$ and Lipschitz when
restricted to $\zeta$-regular triangles. For two adjacent triangles $K =
[x,y,z]$, $L = [x,y,p]$, we have that
\[d_{KL} \geq \frac12 \big| |q(K)-p| -r(K) \big|\,.
  \]
\end{lma}

\begin{proof}
The circumcenter $q = q(K)\in \R^3$ of the triangle $K = [x,y,z]$ is the solution to the linear system
\begin{equation}
\begin{cases}
(q - x)\cdot (y-x) = \frac12 |y-x|^2\\
(q - x)\cdot (z-x) = \frac12 |z-x|^2\\
(q - x)\cdot ((z-x)\times (y-x)) = 0.
\end{cases}
\end{equation}

Thus, the circumcenter map $(x,y,z)\mapsto q$ is $C^1$ when restricted to $\zeta$-regular $K$. To see that the map is globally Lipschitz, it suffices to note that it is $1$-homogeneous in $(x,y,z)$.

For the second point, let $s=q(L)\in \R^3$ be the circumcenter of $L$. Then by the triangle inequality, we have
\begin{equation}
\begin{aligned}
 |p-q|\leq  |p-s| + |s-q| = |x-s| + |s-q| \leq |x-q| + 2|s-q| = r + 2d_{KL},\\
 |p-q| \geq |p-s| - |s-q| = |x-s| - |s-q| \geq |x-q| - 2 |s-q| = r - 2d_{KL}. 
\end{aligned}
\end{equation}
This completes the proof.
\end{proof}

\begin{lma}
\label{lem:char_func}
  Let $\zeta>0$, and $a,b,c,d\in \R^2$  such that $K=[a,b,c]$ and $L=[b,c,d]$ are $\zeta$-regular.  
  \begin{itemize}\item[(i)] 
We have that
     \begin{equation*}
        \int_{\R^2} \mathds{1}_{KL}^v(x)\d x = |v|l_{KL}\theta_{KL}\,.
    \end{equation*}
  \item[(ii)] Let $\delta>0$, $v,w\in\R^2$, $\bar v=(v,v\cdot w)\in \R^3$, 
    $\bar a=(a,a\cdot w)\in \R^3$ and $\bar b, \bar c,\bar d\in \R^3$ defined
    analogously.  Let 
  $\bar K=[\bar a,\bar b,\bar c]$, $\bar L=[\bar b,\bar c,\bar d]$.
 %    Let $\nu_{\bar K\bar L}$ be the unit vector contained in the plane
 % that also contains $\bar a,\bar b,\bar c,\bar d$, pointing from $\bar K=[\bar
 % a,\bar b,\bar c]$ to  $\bar L=[\bar b,\bar c,\bar d]$.
 %    Define 
 %        \[
 %      \theta_{\bar K,\bar L}^{\bar v}:=\frac{|\nu_{\bar K\bar L}\cdot \bar v|}{|\bar
 %        v|}\,.
 %    \]
    Then
    \[
      \int_{\R^2} \mathds{1}_{KL}^v(x)\, \d x = \frac{|\bar v|}{\sqrt{1+|w|^2}}
      \theta_{\bar K\bar L}^{\bar v}l_{\bar K\bar L}\,.
    \]
  \end{itemize}
\end{lma}
\begin{proof}
  The equation (i) follows from the fact that $\mathds{1}_{KL}^v$ is the
  characteristic function of a parallelogram, see Figure \ref{fig:parallelogram}.
To prove (ii) it suffices to observe that
  $\int_{\R^2} \mathds{1}_{KL}^v(x)\sqrt{1+w^2}\d x$ is the volume of the
  parallelogram from (i) pushed forward by the map $\tilde h(x)= (x,x\cdot
  w)$, see Figure \ref{fig:char_fun_2}.
\end{proof}

  \begin{figure}[h]
% \begin{subfigure}{.45\textwidth}
%   \begin{center}
%     \includegraphics[height=5cm]{char_fun_1.pdf}
%   \end{center}
%   \caption{The volume of the parallelogram spanned by $v$ and the shared side of two triangles $K$, $L$. \label{fig:char_fun_1}}
% \end{subfigure}
% \hspace{5mm}
% \begin{subfigure}{.45\textwidth}
\includegraphics[height=5cm]{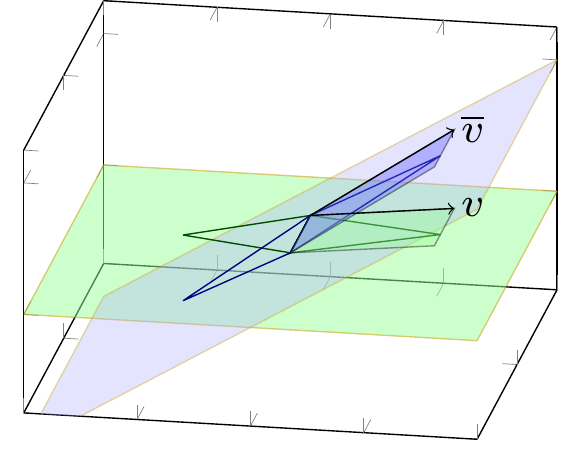}
\caption{The parallelogram pushed forward by an affine map $x\mapsto (x,x\cdot w)$. \label{fig:char_fun_2}} 
%\end{subfigure}
\end{figure}

\subsection{Graphs over manifolds}

\begin{assump}
  \label{ass:Mprop}
We assume $M\subset\R^3$ 
is
 an oriented compact two-dimensional $C^\infty$-submanifold of $\R^3$.
\end{assump}
This manifold will be fixed in the following. We denote the normal of $M$ by $n_M:M\to S^2$, and the second fundamental form at $x_0\in M$ is denoted by $S_M(x_0):T_{x_0}M\to T_{x_0}M$.

\medskip

\begin{defi}
  \label{def:radius_injectivity}
   The \emph{radius of injectivity} $\delta(M)>0$ of $M$ is   the largest number such that the map $\phi:M\times (-\delta(M),\delta(M))\to \R^3$, $(x,h) \mapsto x + h n_M(x)$ is injective and the operator norm of $\delta(M)S_M(x)\in\mathcal{L}(T_xM)$ is at most $1$ at every $x\in M$.
  \end{defi}

We define a graph over $M$ as follows:

\begin{defi}
  \label{def:Mgraph}
  \begin{itemize}
\item[(i)] A set $M_h = \{x+ h(x)n_M(x)\,:\,x\in M\}$ is called a \emph{graph} over $M$ whenever $h:M\to \R$ is a continuous function with $\|h\|_\infty \leq \delta(M)/2$.

\item[(ii)] The graph $M_h$ is called a ($Z$-)Lipschitz graph (for $Z > 0$) whenever $h$ is ($Z$-)Lipschitz, and a smooth graph whenever $h$ is smooth.
\item[(iii)] A set $N\subset B(M,\delta(M)/2)$ is said to be locally a tangent Lipschitz
  graph over $M$ if for every $x_0\in M$ there exists $r>0$ and a Lipschitz
  function $h:(x_0 +T_{x_0}M)\cap B(x_0,r)\to \R$ such that the intersection of $N$
  with the cylinder $C(x_0,r,\frac{\delta(M)}{2})$ over $(x_0 +T_{x_0}M)\cap B(x_0,r)$ with height $\delta(M)$ in both
  directions of $n_M(x_0)$, where
  \[
     C(x_0,r,s) := \left\{x + tn_M(x_0)\,:\,x\in (x_0 + T_{x_0}M) \cap B(x_0,r), t\in [-s,s] \right\},
%    N^{x_0}:=N\cap \left(\left((x_0 + T_{x_0}M)\cap B(x_0,r)\right)+
%    n_M(x_0)\left[-\frac{\delta(M)}{2},\frac{\delta(M)}{2}\right]\right)
  \]
  is equal to the graph of $h$ over $T_{x_0}M\cap B(x_0,r)$,
  \[
    N \cap C\left(x_0,r,\frac{\delta(M)}{2}\right) =\{x+h(x)n_M(x_0):x\in (x_0+T_{x_0}M)\cap B(x_0,r)\}\,.
    \]
\end{itemize}
\end{defi}

\begin{lma}\label{lma: graph property}
  Let $N\subset B(M,\delta(M)/2)$ be locally a tangent Lipschitz graph over $M$.Then $N$ is a Lipschitz graph over $M$.
 
\end{lma}

% {\footnotesize\color{red}
%   Comments:
%   \begin{itemize}\item 
% (i) and (ii) of the previous version are not necessary, since in the lower
%     bound one argues locally and may immediately consider all objects as given
%     graphs. More precisely, the approximating sequence will consist of graphs of
%     functions $h_j\in W^{1,\infty}\cap W^{2,2}$ over flat two-dimensional domains, as will the limit function $h$.
%     \item A quantitative estimate of the Lipschitz bound in the lemma is not
%       necessary either, since in the upper bound one only needs to establish
%       that the triangulations in the plane deliver graphs when pushed forward
%       suitably. It needs to be shown that the corresponding piecewise affine
%       functions $h_j$ converge strongly in $W^{1,\infty}$ to $h\in C^2$; but this follows from the
%       regularity of $h$ and the fact that the control points of the graph of
%       $h_j$ lie on the graph of $h$. (The convergence for $h\in W^{2,2}\cap
%       W^{1,\infty}$ is then done by approximation.)
%     \end{itemize}
%     }

    \begin{proof}
By Definition \ref{def:Mgraph} (iii), we have that for every $x\in M$, there
exists exactly one element
\[
  x'\in
  N\cap \left( x+n_M(x_0)[-\delta(M),\delta(M)]\right)\,.
\]
We write $h(x):=(x'-x)\cdot n_M(x)$, which obviously
implies $N=M_h$. For every $x_0\in
M$ there exists a neighborhood of $x_0$ such that $h$ is Lipschitz continuous in
this neighborhood by
the locally tangent Lipschitz property and the regularity of $M$. The global
Lipschitz property for $h$ follows from the local one by a standard covering argument.
    \end{proof}
% \medskip

% The argument goes like this: Cover $M$ by balls $B_j$, $j=1,\dots,N$ such that
% \[
%   M\subset \bigcup B_j/4\,,
% \]
% and $h$ is Lipschitz on $B_j\cap M$ with constant $L_j$. Let $x,y\in M$. Then at least one  of the following cases holds true:
% \begin{itemize}\item 
%  $\exists j$ such that $x,y\in B_j$ and $|h(x)-h(y)|\leq \max_j L_j|x-y|$
% \item $x\in B_i/4$, $y\in B_j/4$ with $B_i/4\cap B_j/4=\emptyset$ and
%   \[|h(x)-h(y)|\leq \frac{2\max |h|}{\min_{B_i/4\cap B_j/4=\emptyset} \dist(B_i/4,B_j/4)}|x-y|\,.\]
% \end{itemize}

    \begin{lma}
      \label{lem:graph_rep}
      Let $h_j\in W^{1,\infty}(M)$ with $\|h_j\|_{\infty}\leq\delta(M)/2$ and $h_j\wsto h\in W^{1,\infty}(M)$ for $j\to \infty$. Then
      for every point $x\in M$, there exists a neighborhood $V\subset x+T_xM$,
      a Euclidean motion $R$ with $U:=R(x+T_xM)\subset \R^2$,  functions $\tilde
      h_j:U\to\R$ and $\tilde h:U\to \R$ such that $\tilde h_j\wsto \tilde h$
      in $W^{1,\infty}(U)$ and
      \[
        \begin{split}
          R^{-1}\Gr \tilde h_j&\subset M_{h_j} \\
          R^{-1}\Gr \tilde h&\subset M_{h} \,.
        \end{split}
        \]
    \end{lma}
    \begin{proof}
      This follows immediately from our assumption that $M$ is $C^2$ and the
      boundedness of $\|\nabla h_j\|_{L^\infty}$.
    \end{proof}
    
% \subsection{Triangular surfaces as graphs}

%     When we speak of flat triangular surfaces $\cT\subset\R^2$ in the sequel, it is understood
% that $\cT=\{K_i:i=1,\dots,N\}$ is a triangular surface in the sense of
% Definition \ref{def:triangular_surface} with $\cup_{i=1}^N K_i\subset\R^2\times\{0\}$.

    \section{Proof of compactness and lower bound}

\label{sec:proof-comp-lower}

\begin{notation}
  \label{not:push_gen}
  If $U\subset\R^2$, $\cT$ is a flat triangulated surface $\cT\subset U$, $h:U\to\R$ is Lipschitz, and 
  $K=[a,b,c]\in\cT$, then we write 
  \[
    h_*K=[(a,h(a)),(b,h(b)),(c,h(c))]\,.
    \]
  We denote by  $h_*\cT$ for the triangulated surface defined by
  \[
    K\in\cT\quad\Leftrightarrow \quad h_*K\in
    h_*\cT\,.
  \]
\end{notation}

For an illustration  Notation \ref{not:push_gen}, see Figure \ref{fig:push_gen}.

\begin{figure}[h]
\includegraphics[height=5cm]{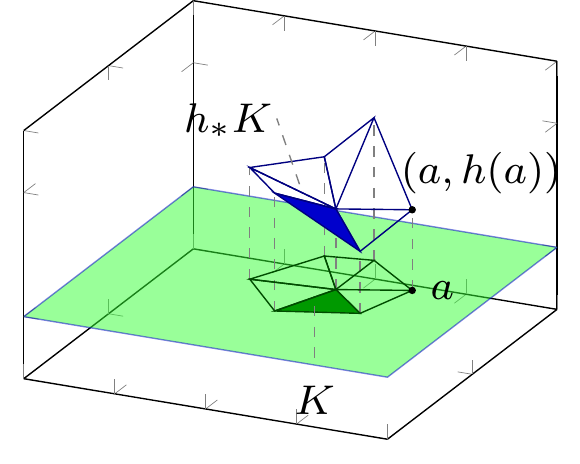}
\caption{Definition of the push forward of a triangulation $\mathcal T\subset \R^2$ by a map $h:\R^2 \to \R$. \label{fig:push_gen}} 
\end{figure}

% In the following lemma, we will consider a flat triangular surface
% $\cT\subset\R^2$ and a real-valued Lipschitz map $h$ whose domain contains
% $\cT$. For each $K=[a,b,c]\in \cT$, we define
% \[
%   K^*=[(a,h(a)),(b,h(b)),(c,h(c))]
% \]
% and define the push forward of the triangular surface by $h$ via

% \[
% h_*\cT:=\{K^*:K\in \cT\}\,.
%   \]

    \begin{lma}
      \label{lem:CS_trick}
      Let $U\subset\R^2$, let $\cT$ be a flat triangulated surface with $U\subset
      \cT\subset\R^2$,  let $h$ be a Lipschitz function $U\to \R$ that is
      affine on each triangle of $\cT$, $\cT^*=h_*\cT$, let $g$ be a function that is constant on each
      triangle of $\cT$,  $v\in \R^2$, $U^v=\{x\in\R^2:[x,x+v]\subset U\}$,  and $W\subset U^v$.
      \begin{itemize}\item[(i)] 
Then
        \[
          \begin{split}
            \int_{W}& |g(x+v)-g(x)|^2\d x\\
            &\leq | v| \left(\sum_{K,L\in\cT}
              \frac{l_{K^*L^*}}{d_{K^*L^*}} |g(K)-g(L)|^2\right) \max_{x\in W}
            \sum_{K,L\in\cT} \mathds{1}^v_{KL}(x) \frac{\theta_{KL}^vl_{KL}d_{K^*L^*}}{l_{K^*L^*}}
              \,,
          \end{split}
        \]
where we have written $K^*=h_*K$, $L^*=h_*L$ for $K,L\in \cT$.
       
        \item[(ii)]
Let $w\in\R^2$, and denote by
 $\bar K$, $\bar L$  the triangles $K,L$ pushed forward by the map
        $x\mapsto (x,x\cdot w)$.
        Then
        \[
          \begin{split}
            \int_{W}& |g(x+v)-g(x)|^2\d x\\
            &\leq \frac{|\bar v|}{\sqrt{1+|w|^2}} \left(\sum_{K,L\in\cT}
              \frac{l_{K^*L^*}}{d_{K^*L^*}} |g(K)-g(L)|^2\right) \max_{x\in W}
            \sum_{K,L\in\cT} \mathds{1}^v_{KL}(x) \frac{\theta_{\bar K\bar
              L}^{\bar v}l_{\bar K\bar L}d_{K^*L^*}}{l_{K^*L^*}}\,.
          \end{split}
        \]
        % where
        % \[
        %   \begin{split}
        %     \theta_{\bar K\bar L}^{\bar v}&=\frac{\nu_{\bar K\bar
        %         L}\cdot \bar v}{|\bar v|}\,.
        %     % c(h,v)&=\sup_{x\in W}\frac{\sqrt{|v|^2+|v\cdot \nabla
        %     % h(x)|^2}}{\sqrt{1+|\nabla h(x)|^2}}\,.
        %   \end{split}
        % \]
      \end{itemize}
      \end{lma}
    \begin{proof}
By the Cauchy-Schwarz inequality, for $x\in W$, we have that
\[
  \begin{split}
    | g(x+v)- g(x)|^2&\leq \left(\sum_{K,L\in \cT} \mathds{1}_{K
        L}^v(x)| g(K)- g(L)|\right)^2\\
    &\leq \left(\sum_{K,L\in \cT}
      \frac{l_{K^*L^*}}{\theta_{KL}^vl_{KL}d_{K^*L^*}}\mathds{1}_{K
        L}^v(x)| g(K)- g(L)|^2\right)\\
    &\qquad \times\left(\sum_{K,L\in \cT}
        \mathds{1}^v_{KL}(x)\frac{\theta^v_{KL}l_{KL}d_{K^*L^*}}{l_{K^*L^*}}\right)\,.
  \end{split}
\]

Using these estimates and Lemma \ref{lem:char_func} (i), we obtain

\begin{equation}
\begin{aligned}
&\int_{U^v} | g(x+v) -  g(x)|^2\,\d x\\
 % \leq &  \int_{A_j^{\delta, v}}\left( \sum_{K,L\in \cT} \mathds{1}_{KL}(x) |\bar
 %   g(K) - \bar g(L)| \right)^2\,\d x\\
 \leq & \int_{U^v} \left( \sum_{K,L\in \cT}
   \mathds{1}^v_{KL}(x)\frac{l_{K^*L^*}}{\theta^v_{KL}l_{KL}d_{K^*L^*}} | g(K) -  g(L)|^2 \right)\\
 &\quad \times
 \left( \sum_{K,L\in \cT} \mathds{1}^v_{KL}(x) \frac{\theta^v_{KL}l_{KL}d_{K^*L^*}}{l_{K^*L^*}} \right) \,\d x\\
 \leq  & |v|\left( \sum_{K,L\in \cT}
   \frac{l_{K^*L^*}}{d_{K^*L^*}} | g(K) -  g(L)|^2 \right) \max_{x\in U^v}  \sum_{K,L\in \cT} \mathds{1}^v_{KL}(x) \frac{\theta^v_{KL}l_{KL}d_{K^*L^*}}{l_{K^*L^*}}\,.
 \end{aligned}
\end{equation}
This proves (i). The claim (ii)  is proved analogously, using  $
\frac{\theta_{\bar K\bar L}^{\bar v}l_{\bar K\bar L}d_{K^*L^*}}{l_{K^*L^*}}$ instead of $
\frac{\theta_{ K L}^{v}l_{ K L}d_{K^*L^*}}{l_{K^*L^*}}$ in the
Cauchy-Schwarz inequality, and then Lemma \ref{lem:char_func} (ii).
\end{proof}

  In the following proposition, we will consider sequences of flat triangulated surfaces
  $\cT_j$ with  $U\subset\cT_j\subset\R^2$ and sequences of Lipschitz functions
  $h_j:U\to \R$. We write $\cT_j^*=(h_j)_*\cT_j$, and for $K\in \cT_j$, we write
  \[
    K^*=(h_j)_*K\,.
      \]

\begin{prop}\label{prop:lower_blowup}
% Let $U\subset\R^2$ be open bounded simply connected and $h\in W^{1,\infty}(U)$.
Let $U,U'\subset\R^2$ be open, $\zeta>0$, $(\cT_j)_{j\in \N}$
a sequence of flat $\zeta$-regular triangulated surfaces   with $U\subset\cT_j\subset
U'$ and  $\mathrm{size} (\cT_j) \to
0$.  Let $(h_j)_{j\in\N}$ be a sequence of Lipschitz functions $U'\to \R$ with
uniformly bounded gradients such that $h_j$ is affine
on each triangle of $\cT_j$ and the triangulated surfaces
$\cT_j^*=(h_j)_*\cT_j$ satisfy the Delaunay property.

\begin{itemize}
  \item[(i)] Assume that
\[
  \begin{split}
    h_j&\wsto h\quad \text{ in } W^{1,\infty}(U')\,,
  \end{split}
\]
and $\liminf_{j\to \infty}  \sum_{K,L\in\cT_j} \frac{l_{K^*L^*}}{d_{K^*L^*}} |n(K^*) - n(L^*)|^2<\infty$. Then $h\in W^{2,2}(U)$.
\item[(ii)] Let $U=Q=(0,1)^2$,  and let $(g_j)_{j\in\N}$ be a sequence of functions $U'\to\R$ such that $g_j$ is constant on
each triangle in $\cT_j$. Assume that
\[
  \begin{split}
    h_j&\to h\quad \text{ in } W^{1,2}(U')\,,\\
    g_j&\to g \quad\text{ in }
    L^2(U')\,,
  \end{split}
\]
where  $h(x)=w\cdot x$ and $g(x)=u\cdot x$ for some $u,w\in \R^2$. 
Then we have
\[
 u^T\left(\mathds{1}_{2\times 2}+w\otimes w\right)^{-1}u \sqrt{1+|w|^2}\leq \liminf_{j\to \infty}  \sum_{K,L\in\cT_j} \frac{l_{K^*L^*}}{d_{K^*L^*}} |g_j(K) - g_j(L)|^2\,.
\]
\end{itemize}
% Here, the left-hand side is defined as
% \[
% \int_{\Gr h} |\nabla_{\Gr h} g|^2\,d\H^2 = \int_{U} \frac{\left|\nabla g \right|^2}{\sqrt{1+|\nabla h|^2}} \,d\L^2\,.
% \]
\end{prop}

\begin{proof}[Proof of (i)]
We write 
\[
E_j:= \sum_{K,L\in \cT_j} \frac{l_{K^*L^*}}{d_{K^*L^*}} |n(K^*)
- n(L^*)|^2 \,.
\]
Fix $v\in B(0,1)\subset\R^2$,  write $U^v=\{x\in\R^2:[x,x+v]\subset U\}$,
and
fix $k\in \{1,2,3\}$.
Define the function $N_j^k:U\to \R^3$ by requiring $N_j^k(x)=n(K^*)\cdot e_k$ for $x\in K\in\cT_j$.
By Lemma \ref{lem:CS_trick} with $g_j=N_j^k$, we
have that
\begin{equation}
\label{eq:11}
  \int_{U^v} |N_{j}^k(x+v) -  N_j^k(x)|^2\,\d x
  \leq |v|
\left(\max_{x\in U^v}  \sum_{K,L\in\cT_j}\mathds{1}^v_{KL}(x) \frac{\theta^v_{KL}l_{KL}d_{K^*L^*}}{l_{K^*L^*}}
              \right) E_j\,.
\end{equation}
Since
$h_j$ is uniformly Lipschitz, there exists a constant $C>0$ such that
\[
  \frac{l_{KL}}{l_{K^*L^*}}
              d_{K^*L^*}<C d_{KL}\,.
\]
We claim that

  \begin{equation}\label{eq:15}
  \begin{split}
    \max_{x\in U^v} \sum_{K,L\in \cT_j} \mathds{1}_{KL}^v(x) \theta_{KL}d_{KL}
        &\lesssim |v|+C\size(\cT_{j})\,.
  \end{split}
  \end{equation}
  
Indeed, let $K_0,\ldots,K_N\in \cT_{j}$ be the sequence of triangles so that there is $i:[0,1]\to \{1,\ldots,N\}$ non-decreasing with $x+tv\in K_{i(t)}$.
 % and by Lemma \ref{lma: circumcenter regularity},
We have that for all pairs $K_i,K_{i+1}\in \cT_{j}$,
\begin{equation}
  \label{eq:12}
\theta_{K_iK_{i+1}} d_{K_iK_{i+1}} = \left|(q(K_{i+1})-q(K_i)) \cdot \frac{v}{|v|}\right| \,,
\end{equation}
which yields the last estimate in \eqref{eq:15}. 
Inserting in \eqref{eq:11} yields
\begin{equation}
  \int_{U^v} |N_{j}^k(x+v) -  N_j^k(x)|^2\,\d x 
  \leq
  C|v|(|v|+C\size(\cT_{j})) E_j\,.
\end{equation}

By passing to the limit $j\to\infty$ and standard difference quotient arguments,
it then follows that the limit $N^k=\lim_{j\to\infty} N_j^k$ is in
$W^{1,2}(U)$. Since $h$ is also in $W^{1,\infty}(U)$ and $(N^k)_{k=1,2,3}=(\nabla
h,-1)/\sqrt{1+|\nabla h|^2}$ is the normal to the graph of $h$, it follows that $h\in W^{2,2}(U)$.
\end{proof}

 \bigskip

\begin{proof}[Proof of (ii)]
We write 
\[
E_j:= \sum_{K,L\in \cT_j} \frac{l_{K^*L^*}}{d_{K^*L^*}} |g_j(K)
- g_j(L)|^2 
\]
and may assume without loss of generality that  $\liminf_{j\to \infty}
E_j<\infty$.

Fix $\delta > 0$. Define the set of bad triangles as
\[
\B_j^\delta := \{K \in \cT_{j}\,:\,\left|\nabla  h_j(K)- w\right| > \delta\}.
\]

Fix $v\in B(0,1)$, and write $Q^v=\{x\in \R^2:[x,x+v]\subset Q\}$.  Define the set of good points as
\[
    A_j^{\delta,v} := \left\{x\in Q^v: \#\{K\in \B_j^\delta\,:
     \,K \cap [x,x+v] \neq \emptyset\} \leq
      \frac{\delta|v|}{\size(\cT_{j})}\right\}.
\]

We claim that
  \begin{equation}\label{eq:17}
  \L^2(Q^v \setminus A_j^{\delta,v}) \to 0\qquad\text{ for } j\to\infty\,.
\end{equation}
Indeed,
let $v^\bot=(-v_2,v_1)$, and let $P_{v^\bot}:\R^2\to v^\bot \R$ denote the projection onto the linear subspace parallel to $v^\bot$. Now by the definition of $A_j^{\delta,v}$, we may estimate
\[
  \begin{split}
  \int_{Q^v}|\nabla h_j-w|^2\d x \gtrsim & \# \mathcal B_j^{\delta} \left(\size \cT_j \right)^2 \delta\\
  \gtrsim & \frac{\L^2(Q\setminus A_j^{\delta,v})}{|v|\size\cT_j}\frac{\delta|v|}{\size \cT_j} \left(\size \cT_j \right)^2 \delta\\
  \gtrsim &\L^2(Q^v \setminus A_j^{\delta,v})\delta^2|v|\,,
\end{split}
  \]
and hence \eqref{eq:17} follows by 
 $h_j\to h$ in $W^{1,2}(Q)$.
For the push-forward of $v$ under the affine map $x\mapsto (x,h(x))$, 
we write
\[
  \bar v= (v,v\cdot w)\in\R^3\,.
\]
Also, for $K=[a,b,c]\in \cT_j$, we write
\[
  \bar K=[(a,a\cdot w),(b,b\cdot w),(c,c\cdot w)]=h_*K\,.
  \]
By Lemma \ref{lem:CS_trick}, we have that
\begin{equation}
\label{eq: difference quotient estimate}
\begin{split}
  \int_{A_j^{\delta, v}} &| g_{j}(x+v) -  g_j(x)|^2\d x \\
  & \leq \frac{|\bar v|}{\sqrt{1+|w|^2}} \left(\max_{x\in A_j^{\delta, v}}
    \sum_{K,L\in \cT_j} \mathds{1}^v_{KL}(x) \frac{\theta_{\bar K\bar L}^{\bar
      v}l_{\bar K\bar L}d_{K^*L^*}}{l_{K^*L^*}}\right) E_j\,.
\end{split}
\end{equation}
We claim that
  \begin{equation}
  \max_{x\in A_j^{\delta, v}}  \sum_{K,L\in \cT_j} \mathds{1}_{KL}(x)
  \frac{\theta_{\bar K \bar L}^{\bar v}l_{\bar K\bar L}d_{K^*L^*}}{l_{K^*L^*}} \leq
  (1+C\delta)\left(|\bar v|+C\size(\cT_j)\right)\,.\label{eq:16}
  \end{equation}
  Indeed,
Let $K_0,\ldots,K_N\in \cT_{j}$ be the sequence of triangles so that there is $i:[0,1]\to \{1,\ldots,N\}$ non-decreasing with $x+tv\in K_{i(t)}$.

For all pairs $K_i,K_{i+1}\in \cT_{j} $ we have
\begin{equation}
\theta_{\bar K_{i}\bar K_{i+1}}^{\bar v}d_{\bar K_i\bar K_{i+1}} = (q(\bar K_{i+1})-q(\bar K_i)) \cdot \frac{\bar v}{|\bar v|} \,.
\end{equation}
Also,  we have that for $K_i,K_{i+1}\in \cT_{j} \setminus
\B_j^\delta$,
  \begin{equation*}
    \begin{split}
    % (q(K_{i+1}^*)-q(K_i^*))  \cdot \frac{\bar v}{|\bar v|}&\geq -C\e\delta \qquad \text{(by Lemma \ref{lem:affinegood})}\\
      \frac{l_{K_i^*K_{i+1}^*}d_{\bar K_i\bar K_{i+1}}}{l_{\bar K_i\bar K_{i+1}}d_{K_i^*K_{i+1}^*}}&\leq 1+C\delta\,.
    \end{split}
\end{equation*}
Hence
\begin{equation}\label{eq: good triangles}
  \begin{split}
    \sum_{i\,:\,\{K_i,K_{i+1}\}\cap \B_k^\delta = \emptyset}&
    \frac{\theta_{\bar K_{i}\bar K_{i+1}}^{\bar v}l_{\bar K_i\bar K_{i+1}}d_{K_i^*K_{i+1}^*}}{l_{K^*_iK_{i+1}^*}}\\
   & \leq (1+C\delta)\sum_{i\,:\,\{K_i,K_{i+1}\}\cap \B_k^\delta = \emptyset}
    \left(\left(q(\bar K_{i+1})-q(\bar K_i)\right)\cdot
      \frac{\bar v}{|\bar v|}\right)\,.% \\
    % & \leq (1+C\delta)\left(\sum_{i\,:\,\{K_i,K_{i+1}\}\cap \B_k^\delta = \emptyset}
    % \left(q(\bar K_{i+1})-q(\bar K_i)\right)\cdot
    %   \frac{\bar v}{|\bar v|}\right)+C\delta|\bar v|\,.
  \end{split}
\end{equation}
% In the last inequality above, we have used that  by the $\zeta$-regularity of
% $\cT_j$ we have
% \[
%   \#\{K\in
%   \cT_j:[x,x+v]\cap K\neq\emptyset\}\leq C |v|\size{\cT_j}^{-1}\,.
% \]
If one of the triangles $K_i,K_{i+1}$ is in $\B_j^\delta$, then we may estimate
\[
  \left|\left(q(\bar K_{i+1})-q(\bar K_i)\right) \cdot \frac{\bar v}{|\bar v|}\right|\leq C\size\cT_j\,.
\]
Since there are few bad triangles along $[x,x+v]$, we have, using $x\in A_j^{\delta,v}$,
\begin{equation}\label{eq: bad triangles}
  \begin{split}
    \sum_{i\,:\,\{K_i,K_{i+1}\}\cap \B_k^\delta \neq \emptyset}&
    \frac{\theta_{\bar K_{i}\bar K_{i+1}}^{\bar v}l_{\bar K_i\bar K_{i+1}}d_{K_i^*K_{i+1}^*}}{l_{K^*_iK_{i+1}^*}}-(q(\bar K_{i+1})-q(\bar K_i)) \cdot \frac{\bar v}{|\bar v|}\\
    &\leq C\#\{K\in
    \B_j^\delta\,:\,K \cap [x,x+v] \neq \emptyset\} \size(\cT_j)
    \\
    &\leq C\delta|\bar v|\,.
  \end{split}
% C(\theta) \delta |v| + \sum_{i\,:\,\{K_i,K_{i+1}\}\cap \B_k^\delta \neq \emptyset} (x_{K_{i+1}'} - x_{K_i'}) \cdot \frac{v}{|v|}.
\end{equation}

Combining \eqref{eq: good triangles} and \eqref{eq: bad triangles} yields
\begin{equation*}
  \begin{split}
    \sum_{i = 0}^{N-1}\frac{\theta_{\bar K_{i}\bar K_{i+1}}^{\bar v}l_{\bar K_i\bar K_{i+1}}d_{K_i^*K_{i+1}^*}}{l_{K^*_iK_{i+1}^*}}&
\leq (1+C\delta)\sum_{i = 0}^{N-1}(q(\bar K_{i+1})-q(\bar K_i)) \cdot \frac{\bar v}{|\bar v|}+C\delta|\bar v|\\
&=     (1+C\delta)(q(\bar K_N) -
    q(\bar K_0)) \cdot \frac{\bar v}{|\bar v|} + C \delta |\bar v| \\
    &\leq (1+C\delta)\left(|\bar v|
      + C\size(\cT_{j})\right).
  \end{split}
\end{equation*}
This proves \eqref{eq:16}.

\medskip

Inserting \eqref{eq:16} in \eqref{eq: difference quotient estimate},  and
passing to the limits $j\to\infty$ and  $\delta\to 0$, we obtain 
%\begin{equation*}
%    |Q^v| |v\cdot u |^2
%    \leq \frac{|\bar v|^2}{\sqrt{1+|w|^2}}\liminf_{j\to \infty}E_j\,.
%      \end{equation*}
%  Replacing $v$ with $\e v$ and sending  $\e\to 0$,   this yields
  \[|v\cdot u |^2
        \leq \frac{|\bar v|^2}{\sqrt{1+|w|^2}}\liminf_{j\to \infty}E_j\,.
    \]

    Now let
    \[
      \underline{u}:=\left(\mathds{1}_{2\times 2},w\right)^T \left(\mathds{1}_{2\times 2}+w\otimes w\right)^{-1}u\,.
    \]
    Then we have $|\underline{u}\cdot \bar v|=|u\cdot v|$ and hence
    \[
      \begin{split}
      |\underline{u}|^2&=\sup_{v\in \R^2\setminus \{0\}}\frac{|\underline{u}\cdot \bar v|^2}{|\bar v|^2}\\
      &\leq \frac{1}{\sqrt{1+|w|^2}}\liminf_{j\to \infty}E_j\,.
      \end{split}
    \]
This proves the proposition.
\end{proof}

\subsection{Proof of compactness and lower bound in Theorem \ref{thm:main}}

\begin{proof}[Proof of  Theorem \ref{thm:main} (o)]
For a subsequence (no relabeling), we have that $h_j\wsto h$ in
$W^{1,\infty}(M)$. By Lemma \ref{lem:graph_rep}, $\cT_j$ may be locally
represented as the graph of a Lipschitz function $\tilde h_j:U\to \R$, and $M_h$
as the graph of a Lipschitz function $\tilde h:U\to \R$, where $U\subset\R^2$
and $\tilde h_j\wsto \tilde h$ in $W^{1,\infty}(U)$.%  For every $z\in M_h$, some neighborhood $V$ of $z$ in $M_h$ may be represented as the graph of a
% function $\tilde h\in W^{1,\infty}(U)$, where $U\subset\R^2$. By
% this we mean that there exists a Euclidean motion $R$ such that $R\Gr \tilde
% h=V$.  By possibly making $U$  smaller,
% we may assume that for all $j\in\N$, $\cT_j$ may be represented as the  graph of some Lipschitz
% function $\tilde h_j:U\to \R$ over $U$, whose Lipschitz constant is bounded
% uniformly in $j$. In this
% way, we may obtain an atlas of $M_h$ consisting of the graphs of
% $W^{1,\infty}$ functions. 

\medskip

It remains to prove that $\tilde h\in W^{2,2}(U)$.  Since the norm of the
gradients are uniformly bounded, 
$\|\nabla \tilde h_j\|_{L^\infty(U)}<C$, we have that the projections of $\cT_j$ to
$U$ are (uniformly) regular flat triangulated surfaces. Hence by Proposition
\ref{prop:lower_blowup} (i), we have that $\tilde h\in W^{2,2}(U)$. 
\end{proof}

\begin{proof}[Proof of Theorem \ref{thm:main} (i)]

Let $\mu_j =  \sum_{K,L\in \cT_j} \frac{1}{d_{KL}} |n(K) - n(L)|^2
\H^1|_{K\cap L}\in \M_+(\R^3)$. Note that either a subsequence of $\mu_j$ converges
narrowly to some $\mu \in \M_+(M_h)$ or there is nothing to show. We will show in
the first case that
  \begin{equation}
  \frac{d\mu}{d\H^2}(z) \geq |Dn_{M_h}|^2(z)\label{eq:7}
  \end{equation}
at $\H^2$-almost every point $z\in M_h$ which implies in particular the lower bound.

By Lemma \ref{lem:graph_rep},  we may reduce the proof to the situation that $M_{h_j}$, $M_h$ are given as
graphs of Lipschitz functions $\tilde h_j:U\to \R$, $\tilde h:U\to \R$
respectively, where  $U\subset \R^2$ is some open bounded set.

We have that $\tilde h_j$ is piecewise
affine on some (uniformly in $j$) regular triangulated surface  $\tilde \cT_j$ that satisfies
\[
  (\tilde h_j)_*\tilde \cT_j=\cT_j\,.
  \]
% For every $z\in M_h$, some neighborhood $V$ of $z$ in $M_h$ may be represented as the graph of a
% function $\tilde h\in W^{1,\infty}\cap W^{2,2}(U)$, where $U\subset\R^2$. By
% this we mean that there exists a Euclidean motion $R$ such that $R\Gr \tilde
% h=V$. In the sequel, we assume $R=\id_{\R^3}$.
Writing down the surface normal to $M_h$ in the coordinates of $U$,
\[N(x)=\frac{(-\nabla \tilde h, 1)}{\sqrt{1+|\nabla \tilde h|^2}}\,,
  \]
  we have that almost every $x\in U$ is a Lebesgue point of $\nabla N$.
% Obviously, this is equivalent with saying that  $\H^2$-almost every
%   $z=(x,\tilde h(x))\in  V=\Gr \tilde h$ is a Lebesgue point of $Dn_{M_h}$.
We write $N^k=N\cdot e_k$ and note that \eqref{eq:7} is equivalent to
  \begin{equation}
    \label{eq:8}
    \frac{\d\mu}{\d\H^2}(z)\geq \sum_{k=1}^3\nabla N^k(x)\cdot \left(\mathds{1}_{2\times
        2}+\nabla \tilde h(x)\otimes\nabla \tilde h(x)\right)^{-1}\nabla
    N^k(x)\,,
  \end{equation}
where  $z=(x,\tilde h(x))$.
Also, we define $N_j^k:U\to\R^3$ by letting $N_j^k(x)=n((\tilde h_j)_*K)\cdot
e_k$  for $x\in K\in
\tilde {\mathcal T_j}$. (We recall that $n((\tilde h_j)_*K)$ denotes the
normal of the triangle $(\tilde h_j)_*K$.)
  
\medskip  

% After possibly reducing $U$, we may assume that for every $j\in \N$, $\mathcal T_j\cap (U\times\R)$ is
% the graph of some Lipschitz function $\tilde h_j:U\to \R$ that 

% Let $x_0 = y_0 + h(y_0)n_M(y_0)\in M_h$, with $y_0$ a Lebesgue point of $\nabla h$ and $\nabla g$. Then for $r>0$ small enough and $j\in \N$ large enough, by Lemma \ref{lma: graph property}, we have that a subtriangulation of $\cT_j$ is homeomorphic to a triangulation $\overline{\cT_j}$ of $B(0,r) \subset T_{x_0}M_h )$, since both $\cT_j$ and $B(0,r)\subset T_{x_0}M_h$ are local graphs over $T_{y_0}M$.

% Let $\overline{g_j}:B(0,r)\subset T_{x_0}M_h \to \R$ be the piecewise constant interpolation of $g_j$ on $\overline{\cT_j}$. Then $\overline{g_j}(x) \to g(x_0+x)$ in $L^2(B(0,r) \subset T_{x_0}M_h )$.

Let now $x_0\in U$ be a Lebesgue point of $\nabla \tilde h$ and $\nabla N$.
We write $z_0=(x_0,\tilde h(x_0))$.
Combining the narrow convergence $\mu_j\to\mu$ with the Radon-Nikodym differentiation Theorem, we may  choose a sequence $r_j\downarrow 0$ such that
\[
  \begin{split}
    r_j^{-1}\size{\cT_j}&\to 0\\
    \liminf_{j\to\infty}\frac{\mu_j(Q^{(3)}(x_0,r_j))}{r_j^2}&= \frac{\d\mu}{\d\H^2}(z_0)\sqrt{1+|\nabla \tilde h(x_0)|^2}\,,
  \end{split}
\]
where $Q^{(3)}(x_0,r_j)=x_0+[-r_j/2,r_j/2]^2\times \R$ is the cylinder over $Q(x_0,r_j)$.

Furthermore,  let $\bar N_j,\bar h_j,\bar N,\bar h: Q\to \R$  be defined by 
\[
  \begin{split}
    \bar N_j^k(x)&=\frac{N_j(x_0+r_j x)-N_j(x_0)}{r_j}\\
    \bar N^k(x)&=\nabla N^k(x_0)\cdot (x-x_0)\\
    \bar h_j(x)&=\frac{\tilde h_j(x_0+r_j x)-\tilde h_j(x_0)}{r_j}\\
    \bar h(x)&=\nabla \tilde h(x_0)\cdot (x-x_0)\,.
  \end{split}
\]
We recall that by assumption we have that $N^k\in W^{1,2}(U)$. This
implies in particular that (unless $x_0$ is contained in a certain set of
measure zero, which we discard), we have that
  \begin{equation}\label{eq:9}
  \bar N_j^k\to \bar N^k\quad\text{ in } L^2(Q)\,.
  \end{equation}
%   without increasing
% $\liminf_{j\to\infty}\frac{\mu_j(Q^{(3)}(x_0,r_j))}{r_j^2}$.
Also, let $T_j$ be the blowup map
\[
  T_j(x)=\frac{x-x_0}{r_j}
\]
and let $\cT_j'$ be the triangulated surface one obtains by blowing up $\tilde\cT_j$,
defined by
\[
  \tilde K\in \tilde \cT_j\quad \Leftrightarrow \quad T_j\tilde K \in \cT_j'\,.
  \]
Now let $\mathcal S_j$ be the smallest subset of  $\cT_j'$ (as sets of
triangles) such that
$Q\subset\mathcal S_j$ (as subsets of $\R^2$). 
Note that $\size\mathcal S_j\to 0$,  $\bar N_j^k$ is constant  and $\bar h_j$ is
affine on each $K\in \mathcal S_j$. Furthermore, for $x\in K\in \tilde \cT_j$, we have that
\[
  \nabla \tilde h_j(x)=\nabla \bar h_j(T_jx)
\]
This implies in particular
  \begin{equation}
  \bar h_j\to \bar h\quad \text{ in } W^{1,2}(Q)\,.\label{eq:6}
  \end{equation}
Concerning the discrete energy functionals, we have for the rescaled
triangulated surfaces $(\cT_j')^*=(\bar h_j)_* \cT_j'$, with $K^*=(\bar h_j)_*K$
for $K\in \cT_j'$,
  \begin{equation}\label{eq:10}
\liminf_{j\to\infty} \sum_{K,L\in
    \cT_j'}\frac{l_{K^*L^*}}{d_{K^*L^*}} |\bar N_j(K)-\bar N_j(L)|^2\leq  \liminf_{j\to\infty}r_j^{-2}\mu_j(Q^{(3)}(x_0,r_j)) \,.
\end{equation}

Thanks to \eqref{eq:9}, \eqref{eq:6}, we may  apply Proposition
\ref{prop:lower_blowup} (ii) to the sequences of functions $(\bar h_j)_{j\in\N}$,
$(\bar N_j^k)_{j\in\N}$. This yields (after summing over $k\in\{1,2,3\}$)
\[
  \begin{aligned}
    |Dn_{M_h}|^2(z_0)&\sqrt{1+|\nabla \tilde
      h(x_0)|^2}\\
   & = \nabla N(x_0)\cdot \left(\mathds{1}_{2\times 2} +\nabla \tilde h(x_0)\otimes
      \nabla \tilde h(x_0)\right)^{-1}\nabla N(x_0)\sqrt{1+|\nabla \tilde
      h(x_0)|^2} \\
   &  \leq  \liminf_{j\to\infty} \sum_{K, L\in
    \cT_j'}\frac{l_{K^*L^*}}{d_{K^*L^*}} |\bar N_j(K)-\bar N_j(L)|^2\,,
  \end{aligned}
  \]

which in combination with \eqref{eq:10} yields \eqref{eq:8} for $x=x_0$, $z=z_0$ 
and completes the proof
of the lower bound.
\end{proof}

\section{Surface triangulations and upper bound}

\label{sec:surf-triang-upper}

% The goal in this section is to prove the upper bound:

% \begin{prop}\label{prop: upper bound}
%     There exists a universal constant $\zeta_0>0$ such that the following holds:
%     Let $M\subset \R^3$ be a $C^2$-manifold homeomorphic to $S^2$. Let $M_h$ be a $Z$-Lipschitz graph over $M$, with $h\in W^{2,2}(M)$.Then there exists a sequence $(\cT_j)_j$ of $\zeta_0$-regular Delaunay triangulated surfaces so that $\size \cT_j \to 0$ and $\cT_j = M_{h_j}$ for $h_j \to h$ in $W^{1,\infty}(M)$. In addition
%     \begin{align}
%     \lim_{j\to \infty} \frac{1}{2} \sum_{K,L \in \cT_j} \frac{\l{K}{L}}{d_{KL}} |n(K) - n(L)|^2 = \int_{M_h} |D n|^2 \,d\Hm^2.
%     \end{align}
% \end{prop}

Our plan for the construction of a recovery sequence is as follows:  We shall construct optimal sequences of triangulated surfaces first locally around a point $x\in M_h$. It turns out the optimal triangulation must be aligned with the principal curvature directions at $x$. By a suitable covering of $M_h$, this allows for an approximation of the latter in these charts (Proposition \ref{prop: local triangulation}). We will then formulate sufficient conditions for a vertex set to supply a global approximation (Proposition \ref{prop: Delaunay existence}). The main work that remains to be done at that point to obtain a proof of Theorem \ref{thm:main} (ii) is  to add vertices to the local approximations obtained from Proposition \ref{prop: local triangulation} such that the conditions of Proposition \ref{prop: Delaunay existence} are fulfilled.

% Now we prove that an analogous result still holds for the triangulation pushed forward by  a Lipschitz map $h:\cT\to \R$ with small Lipschitz constant.

% \begin{lma}
%   Let $X,\cT,\Omega,\delta$ be as in Theorem \ref{thm: planar Delaunay}, and $h:\Omega\to \R$ Lipschitz with $\|\nabla h\|_\infty<\delta/2$. Furthermore let
%   \[
%     \begin{split}
%     X^*&=\{(x,h(x)):x\in X\}\\
% K^*&=h_*(K) \text{ for } K\in \cT\\
% \cT^*&=h_*(\cT)\,.
% \end{split}
% \]
% Then $\cT^*$ is a Delaunay regular, $\cT^*$ is homeomorphic to $\cT$ as two-dimensional manifolds, and for all $K^*,L^*\in \cT^*$, we have that
% $d_{K^*L^*} \geq \frac{(1-C\delta)\delta}{2} D(X^*)$. 
% \end{lma}

% \begin{proof}
%   The homeomorphy is clear.
%   We observe that for all $x,y\in \Omega$, we have that
%   \[
%     (1-C\delta)|x-y|\leq |(x,h(x))-(y,h(y))|\leq (1+C\delta) |x-y|\,,
%   \]
%   The estimate of $d_{K^*L^*}$ is then seen to be a consequence of Lemma \ref{lma: circumcenter regularity}. We prove that $\cT^*$ possesses the Delaunay property:

% \end{proof}

\subsection{Local optimal triangulations}

\begin{prop}\label{prop: local triangulation}
There are constants $\delta_0, C>0$ such that for all $U \subset \R^2$ open, convex, and bounded; and $h\in C^3(U)$ with $\|\nabla h\|_\infty \eqqcolon  \delta \leq \delta_0$, the following holds:

Let $\eps > 0$, $ C\delta^2 < |\theta| \leq \frac12$, and define $X \coloneqq \{(\eps k + \theta \eps l , \eps l, h(\eps k + \theta \eps l, \eps l))\in U\times \R\,:\,k,l\in \Z\}$. Then any Delaunay triangulated surface $\cT$ with vertex set $X$ and maximum circumradius $\max_{K\in \cT} r(K) \leq \eps$ has
\begin{equation}\label{eq: local error}
\begin{aligned}
 \sum_{K,L\in \cT}& \frac{\l{K}{L}}{d_{KL}}|n(K) - n(L)|^2\\
\leq &\left(1+ C(|\theta|+\delta+\e)\right) \L^2(U) \times\\
&\times\left(\max_{x\in U} |\partial_{11} h(x)|^2 + \max_{x\in U} |\partial_{22} h(x)|^2 + \frac{1}{|\theta|} \max_{x\in U} |\partial_{12} h(x)|^2  \right)+C\e\,.
\end{aligned}
\end{equation}
\end{prop}

\begin{proof}
  We assume without loss of generality that $\theta > 0$.
 We consider the projection of $X$ to the plane,
  \[
    \bar X:=\{(\eps k + \theta \eps l , \eps l)\in U:k,l\in\Z\}\,.
  \]
  Let $\bar\cT$ be the flat triangulated surface that consists of the triangles of the form
  \[
    \begin{split}
    \e[ ke_1+l(\theta e_1+e_2),(k+1)e_1+l(\theta e_1+e_2),ke_1+(l+1)(\theta e_1+e_2)]\\
    \text{ or } \quad \e[ ke_1+l(\theta e_1+e_2),(k+1)e_1+l(\theta e_1+e_2),ke_1+(l-1)(\theta e_1+e_2)]\,,
  \end{split}
  \]
  with $k,l\in \Z$ such that the triangles are contained in $U$, see Figure \ref{fig:upper2d_barT}.

  \begin{figure}[h]
    \centering
\includegraphics[height=5cm]{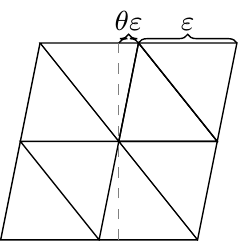}
    \caption{The flat triangulated surface $\bar \cT$. \label{fig:upper2d_barT}}
  \end{figure}
Obviously the flat triangulated surface $\bar\cT$ 
has vertex set $\bar X$. Also, we have that
    \begin{equation}\label{eq:19}
    |x-y|\leq |(x,h(x))-(y,h(y))|\leq (1+C\delta)|x-y|
  \end{equation}
  for all $x,y \in \bar X$. We claim that  for $\delta$ chosen small enough, we have the implication
    \begin{equation}\label{eq:18}
h_*K=[(x,h(x)),(y,h(y)),(z,h(z))]\in \cT\quad \Rightarrow  \quad K= [x,y,z]\in \bar\cT \,.
\end{equation}
Indeed, if $K\not\in  \bar  \cT$, then either $r(K)>\frac32\e$ or there exists $w\in X$ with $|w-q(K)|<(1 -C\theta)r(K)$. In the first case,  $r(h_*K)>(1-C\delta)\frac32\e$ by \eqref{eq:19}  and hence  $h_*K\not\in \cT$ for $\delta$ small enough.  In the second case, we have  by \eqref{eq:19} and Lemma \ref{lma: circumcenter regularity} that
\[
  |(w,h(w))-q(h_*K)|<(1+C\delta)(1 -C\theta)r(h_*K)\,,
\]
and hence $h_*K$ does not satisfy the Delaunay property for $\delta$ small enough. This proves \eqref{eq:18}.
  % Indeed, if $h_*K\in \T$ with $r(h_*K)\leq \e$, we have that $K=[x,y,z]\in\bar \cT$ if and only if for every $w\in \bar \cT$, we have that
  % \[
  %   |w-q(K)|>(1+C\theta)r(K)\,.
  % \]
  % By \eqref{eq:19} and Lemma \ref{lma: circumcenter regularity}, this is equivalent to 
  %   \begin{equation}\label{eq:20}
  %   |(w,h(w))-q(h_*K)|>(1+C_1\theta-C_2\delta))r(h_*K)\,,
  % \end{equation}
  % If the latter holds true, then 
  % the  circumball of $h_*K$ does not contain any elements of $X$, which implies that $h_*K\in \cT$. On the other hand, if $h_*K\in \cT$, then there is no $w\in \bar X$ such that $(w,h(w))$ is contained in the circumball of $h_*K$. Assume that $K$ is not contained in $\bar \cT$. Then there exists $w\in \bar X$ that is contained in the circumball of $K$,
  % \[
  %   |w-q(K)|<r(K)\,.
  %   \]
  %   By \eqref{eq:19} and Lemma \ref{lma: circumcenter regularity}, this yields
  %     \[
  %   |w-q(K)|<(1+C\delta)r(h_*K)\,.
  %   \]

Let $[x,y]$ be an edge with either $x,y \in X$ or $x,y \in \bar X$. We call this edge \emph{horizontal} if $(y-x) \cdot e_2 = 0$, \emph{vertical} if $(y-x) \cdot (e_1 - \theta e_2)= 0$, and \emph{diagonal} if $(y-x) \cdot (e_1 + (1-\theta) e_2) = 0$.
By its definition, $\bar \cT$ consists only of triangles with exactly one horizontal, vertical, and diagonal edge each. By what we have just proved, 
the same is true for $\cT$.

%Since the length of an edge of $K\in \cT$ is at most $2\eps$, theoretically also allowing double horizontal edges, if $[x, x+2\eps e_1]\subset K \in \cT$, then $q(K) = x + \eps e_1$. If $\delta <  \frac{\sqrt{3}}{2}$, the point $y\in K$ with $(y-x) \cdot e_1 = \eps$ will be in $B(q(K),\eps)$. It follows that every triangle must have exactly one edge of each type.

% Further, we note by the same reasoning that the orthogonal projections of $K\in \cT$ onto $\R^2$ form a subcomplex of the planar Delaunay triangulation of the projection of $X$.

\medskip

To calculate the differences between normals of adjacent triangles, let us consider one fixed triangle $K\in \cT$ and its neighbors  $K_1,K_2,K_3$, with which $K$ shares a horizontal, diagonal and vertical  edge  respectively, see Figure \ref{fig:upper2d}.

\begin{figure}[h]
\includegraphics[height=5cm]{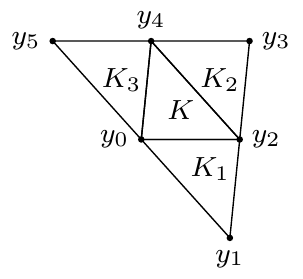}
\caption{Top view of a triangle $K\in\cT$ with its  horizontal, diagonal and vertical  neighbors $K_1,K_2,K_3$. \label{fig:upper2d}} 
\end{figure}

We  assume without loss of generality that one of the vertices of $K$ is the origin. We write
$x_0=(0,0)$, $x_1=\e(1-\theta,-1)$, $x_2=\e(1,0)$, $x_3=\e(1+\theta,1)$, $x_4=\e(\theta,1)$, $x_5=\e(\theta-1,1)$, and $y_i=(x_i, h(x_i))$ for $i=0,\dots,5$. With this notation we have $K=[y_0,y_2,y_4]$, $K_1=[y_0,y_1,y_2]$, $K_2=[y_2,y_3,y_4]$ and $K_3=[y_4,y_5,y_0]$. See Figure \ref{fig:upper2d_barT}.  As approximations of the normals, we define
\[
  \begin{split}
  v(K)&=\e^{-2}y_2\wedge y_4\,\\
  v(K_1)&=\e^{-2} y_1\wedge y_2\\
  v(K_2)&= \e^{-2}(y_3-y_2)\wedge(y_4-y_2)\\
  v(K_3)&= \e^{-2} y_4\wedge y_5\,.
\end{split}
  \]
Note that $v(L)$ is parallel to $n(L)$ and $|v(L)|\geq 1$ for $L\in \{K,K_1,K_2,K_3\}$. 

Hence for $i=1,2,3$, we have that
\[
  |n(K)-n(K_i)|^2\leq |v(K)-v(K_i)|^2\,.
\]
For each $x_i$, we write
\[
  h(x_i)= x_i \cdot \nabla h(0) + \frac12 x_i \nabla^2 h(0) x_i^T+O(\e^3)\,,
\]
where $O(\e^3)$ denotes terms $f(\e)$ that satisfy
$\limsup_{\e\to 0}\e^{-3}|f(\e)|<\infty$.
By an explicit computation we obtain that
\[
  \begin{split}
    \left|v(K)-v(K_1)\right|^2&= \e^2\left|(\theta-1)\theta \partial_{11} h+2(\theta-1)\partial_{12}h+\partial_{22}h\right|^2+O(\e^3)\\
    \left|v(K)-v(K_2)\right|^2&= \e^2\left(\left|\theta\partial_{11} h+\partial_{12}h\right|^2+\left|(\theta-1)\theta\partial_{11} h+(\theta-1)\partial_{12}h\right|^2\right)+O(\e^3)\\
    \left|v(K)-v(K_3)\right|^2&=\e^2\left( \theta^2\left|(\theta-1)\partial_{11} h+\partial_{12}h\right|^2+\left|(\theta-1)\partial_{11} h+\partial_{12}h\right|^2\right)+O(\e^3)\,,
\end{split}
\]
where all derivatives of $h$ are taken at $0$.
Using the Cauchy-Schwarz inequality and $|1-\theta|\leq 1$, we may estimate the term on the right hand side in the first line above,
\[
  \left|(\theta-1)\theta\partial_{11} h+2(\theta-1)\partial_{12}h+\partial_{22}h\right|^2
\leq (1+\theta) |\partial_{22}h|^2+ \left(1+\frac{C}{\theta}\right)\left(\theta^2 |\partial_{11} h|^2+|\partial_{12}h|^2\right)\,.
\]
In a similar same way, we have
\[
  \begin{split}
  \left|\theta\partial_{11} h+\partial_{12}h\right|^2+\left|(\theta-1)\theta\partial_{11} h+(\theta-1)\partial_{12}h\right|^2&\leq C(|\partial_{12}h|^2
  +\theta^2|\partial_{11} h|^2)\\
  \theta^2\left|(\theta-1)\partial_{11} h+\partial_{12}h\right|^2+\left|(\theta-1)\partial_{11} h+\partial_{12}h\right|^2&\leq (1+\theta)|\partial_{11} h|^2+\frac{C}{\theta}|\partial_{12}h|^2\,,
\end{split}
  \]
so that
\[
  \begin{split}
    \left|n(K)-n(K_1)\right|^2&\leq \e^2(1+\theta) |\partial_{22}h|^2+ C\e^2 \left(\theta |\partial_{11} h|^2+ \frac1\theta |\partial_{12}h|^2\right)+O(\e^3)\\
    \left|n(K)-n(K_2)\right|^2&\leq C\e^2 (|\partial_{12}h|^2
  +\theta^2|\partial_{11} h|^2)+O(\e^3)\\
    \left|n(K)-n(K_3)\right|^2&\leq \e^2(1+\theta)|\partial_{11} h|^2+\frac{C}{\theta}\e^2|\partial_{12}h|^2+O(\e^3)\,,
\end{split}
\]

  Also, we have by Lemma \ref{lma: circumcenter regularity} that
\[
  \begin{split}
  \frac{l_{KK_1}}{d_{KK_1}}&\leq 1+C(\delta+\e+\theta)\\
  \frac{l_{KK_2}}{d_{KK_2}}&\leq \left(1+C(\delta+\e+\theta)\right) \frac{C}{\theta}\\
  \frac{l_{KK_3}}{d_{KK_3}}&\leq 1+C(\delta+\e+\theta)\,.
\end{split}
\]
Combining all of the above, and summing up over all triangles in $\cT$, we obtain the statement of the proposition.
\end{proof}

\subsection{Global triangulations}

We are going to use a known fact about triangulations of point sets in $\R^2$, and transfer them to $\R^3$. We first cite a result for planar Delaunay triangulations, Theorem \ref{thm: planar Delaunay} below, which can be found in e.g. \cite[Chapter 9.2]{berg2008computational}. This theorem states the existence of a Delaunay triangulated surface associated to a \emph{protected} set of points.

\begin{defi}
  Let $N\subset\R^3$ be compact, $X\subset N$  a finite set of points and
  \[
    D(X,N)=\max_{x\in N}\min_{y\in X}|x-y|\,.
  \]
  We say that $X$ is $\bar \delta$-protected if whenever $x,y,z \in X$ form a regular triangle $[x,y,z]$ with circumball $\overline{B(q,r)}$ satisfying $r \leq D(X,N)$, then $\left| |p-q| - r \right| \geq \bar\delta$ for any $p\in X \setminus \{x,y,z\}$.
\end{defi}

\begin{thm}\label{thm: planar Delaunay}[\cite{berg2008computational}]
  Let $\alpha > 0$. Let $X\subset \R^2$ be finite and not colinear. Define $\Omega := \conv(X)$. Assume that
  \[\min_{x\neq y \in X} |x-y| \geq \alpha D(X,\Omega)\,,
  \]
  and that $X$ is $\delta D(X,\Omega)$-protected for some $\delta>0$. Then there exists a unique maximal Delaunay triangulated surface $\cT$ with vertex set $X$, given by all regular triangles $[x,y,z]$, $x,y,z\in X$, with circumdisc $\overline{B(q,r)}$ such that $B(q,r) \cap X = \emptyset$.

  The triangulated surface $\cT$ forms a partition of $\Omega$, in the sense that
  \[
    \sum_{K\in \cT} \mathds{1}_K = \mathds{1}_\Omega\quad \Hm^2\text{almost everywhere}\,,
  \]
  where $\mathds{1}_A$ denotes the characteristic function of $A\subset \R^3$.
  Further, any triangle $K\in \cT$ with $\dist(K,\partial \Omega) \geq 4D(X,\Omega)$ is $c(\alpha)$-regular, and $d_{KL} \geq \frac{\delta}{2} D(X,\Omega)$ for all pairs of triangles $K \neq L \in \cT$.
\end{thm}

% \begin{rem}
%   \label{rem:col}
%   We may omit the condition that $X$ does not contain colinear points if all   triangles $[x,y,z]$ with  $x,y,z\in X$ and  circumdisc $\overline{B(q,r)}$ such that $B(q,r) \cap X = \emptyset$ are regular triangles.
% \end{rem}

We are now in position to formulate sufficient conditions for a vertex set to yield a triangulated surface that serves our purpose.

\begin{prop}\label{prop: Delaunay existence}
Let $N\subset\R^3$ be a 2-dimensional compact smooth manifold, and let $\alpha, \delta > 0$. Then there is $\eps = \eps(N,\alpha,\delta)>0$ such that whenever $X\subset N$ satisfies
\begin{itemize}
\item [(a)]$D(X,N)  \leq \eps$,
\item [(b)] $\min_{x,y\in X} |x-y| \geq \alpha D(X,N)$,
\item [(c)] $X$ is $\delta D(X,N)$-protected;% , i.e. whenever $x,y,z \in X$ form a true triangle $[x,y,z]$ with circumball $\overline{B(q,r)}$ satisfying $r \leq D(X,N)$, then $\left| |p-q| - r \right| \geq \delta D(X,N)$ for any $p\in X \setminus \{x,y,z\}$.
\end{itemize}

then there exists a triangulated surface $\cT(X,N)$ with the following properties:
\begin{itemize}
\item [(i)] $\size(\cT(X,N)) \leq 2D(X,N)$.
\item [(ii)] $\cT(X,N)$ is $c(\alpha)$-regular.
\item [(iii)] $\cT(X,N)$ is Delaunay.
\item [(iv)] Whenever $K\neq L \in \cT(X,N)$, we have $d_{KL} \geq \frac{\delta}{2} D(X,N)$.
\item [(v)] The vertex set of $\cT(X,N)$ is $X$.
\item [(vi)] $\cT(X,N)$ is a $C(\alpha, N)D(X,N)$-Lipschitz graph over $N$. In particular, $\cT(X,N)$ is homeomorphic to $N$.
\end{itemize}
\end{prop}

The surface case we treat here  can be viewed as a perturbation of Theorem \ref{thm: planar Delaunay}. We note that the protection property (c) is vital to the argument. A very similar result to Proposition \ref{prop: Delaunay existence} was proved in \cite{boissonnat2013constructing}, but we present a self-contained proof here.

\begin{proof}[Proof of Proposition \ref{prop: Delaunay existence}]
We construct the triangulated surface $\cT(X,N)$ as follows: Consider all regular triangles $K=[x,y,z]$ with $x,y,z\in X$ such that the Euclidean Voronoi cells $V_x,V_y,V_z$ intersect in $N$, i.e. there is $\tilde q \in N$ such that $|\tilde q - x| = |\tilde q - y| = |\tilde q - z| \leq |\tilde q - p|$ for any $p\in X\setminus \{x,y,z\}$.

\emph{Proof of (i):} Let $[x,y,z]\in \cT(X,N)$. Let $\tilde q \in V_x \cap V_y \cap V_z \cap N$, set $\tilde r := |\tilde q - x|$. Then $\tilde r = \min_{p\in X} |\tilde q - p| \leq D(X,N)$, and because $[x,y,z]\subset \overline{B(\tilde q, \tilde r)}$ we have $\diam([x,y,z])\leq 2 \tilde r \leq 2D(X,N)$.

\emph{Proof of (ii):} Let $\overline{B(q,r)}$ denote the Euclidean circumball of $[x,y,z]$. Then $r\leq \tilde r$ by the definition of the circumball. Thus $\min(|x-y|,|x-z|,|y-z|) \geq \alpha r$, and $[x,y,z]$ is $c(\alpha)$-regular by the following argument: Rescaling such that $r = 1$, consider the class of all triangles $[x,y,z]$ with $x,y,z \in S^1$, $\min(|x-y|,|x-z|,|y-z|) \geq \alpha$. All these triangles are $\zeta$-regular for some $\zeta>0$, and by compactness there is a least regular triangle in this class. That triangle's regularity is $c(\alpha)$.

\emph{Proof of (iii):} Because of (ii), $N\cap \overline{B(q,r)}$ is a $C(\alpha, N)\eps$-Lipschitz graph over a convex subset $U$ of the plane $ x + \R(y-x) + \R(z-x)$, say $N\cap \overline{B(q,r)} = U_h$. It follows that $\tilde q - q = h(\tilde q) n_U$. Because $h(x)= 0$, it follows that $|\tilde q - q| = |h(\tilde q)| \leq C(\alpha, N) D(X,N)^2$.

Thus, for $D(X,N) \leq \delta(2C(\alpha,N))^{-1}$, we have that $|\tilde q - q| \leq \frac{\delta}{2}D(X,N)$. This together with (c) suffices to show the Delaunay property of $\cT(X,N)$: Assume there exists $p\in X \setminus \{x,y,z\} \cap B(q,r)$. Then by (c) we have $|p-q| \leq r - \delta D(X,N)$, and by the triangle inequality $|p-\tilde q| \leq |p- q| + \frac{\delta}{2}D(x,N) < \tilde r$, a contradiction.

\emph{Proof of (iv):}
It follows also from (c) and Lemma \ref{lma: circumcenter regularity} that for all adjacent $K,L\in \cT(X,N)$ we have $d_{KL} \geq \frac{\delta}{2} D(X,N)$.

\emph{Proof of (v) and (vi):} Let $\eta>0$, to be fixed later. There is $s>0$ such that for every $x_0\in N$, the orthogonal projection $\pi:\R^3 \to x_0 + T_{x_0}N$ is an $\eta$-isometry when restricted to $N\cap B(x_0,s)$, in the sense that that $|D\pi - \id_{TN}|\leq \eta$.

Let us write $X_\pi=\pi(X\cap B(x_0,s))$. This point set fulfills all the requirements of Theorem \ref{thm: planar Delaunay} (identifying $x_0+T_{x_0}N$ with $\R^2$), except for possibly protection.
We will prove below that
  \begin{equation}\label{eq:23}
  X_\pi\text{ is } \frac{\delta}{4}D(X,N) \text{-protected}.
\end{equation}
We will then consider the planar Delaunay triangulated surface $\cT' \coloneqq \cT(X_\pi, x_0 + T_{x_0}N)$,  and show that for $x,y,z\in B(x_0,s/2)$ we have
  \begin{equation}\label{eq:22}
  K:=[x,y,z]\in \cT(X,N)\quad \Leftrightarrow \quad K_\pi:=[\pi(x),\pi(y),\pi(z)]\in \cT'\,
\end{equation}
If we prove these claims, then (v) follows from Theorem \ref{thm: planar Delaunay}, while (vi) follows from Theorem \ref{thm: planar Delaunay} and Lemma \ref{lma: graph property}.

\medskip

We first  prove \eqref{eq:23}: Let $\pi(x),\pi(y),\pi(z)\in X_\pi$, write  $K_\pi= [\pi(x),\pi(y),\pi(z)]$, and assume $r(K_\pi)\leq D(X_\pi,\mathrm{conv}(X_{\pi}))$. For a contradiction, assume that $\pi(p)\in X_\pi\setminus \{\pi(x),\pi(y),\pi(z)\}$ such that
\[
  \left||q(K_\pi)-\pi(p)|-r(K_\pi)\right|<\frac{\delta}{4}D(X,N)\,.
\]
Using again $|D\pi-\id_{TN}|<\eta$ and Lemma \ref{lma: circumcenter regularity},
we obtain, with $K=[x,y,z]$,
\[
  \left||q(K)-p|-r(K)\right|<(1+C\eta)\frac{\delta}{4}D(X,N)\,.
\]
Choosing $\eta$ small enough, we obtain a contradiction to (c). This completes the proof of \eqref{eq:23}.

\medskip

Next we show  the  implication $K\in \cT\Rightarrow K_\pi\in \cT'$: Let $p\in X\cap B(x_0,s) \setminus \{x,y,z\}$. % Denote by $\overline{B(q',r')}$ the circumdisc of $[\pi(x),\pi(y),\pi(z)]$ in the tangent plane $x_0 + T_{x_0}N$, 
Assume for a contradiction that $\pi(p)$ is contained in the circumball of $K_\pi$,
\[
|\pi(p) - q(K_\pi)|\leq r(K_\pi)\,.
\] Then by $|D\pi-\id_{TN}|<\eta$ and Lemma \ref{lma: circumcenter regularity}\,,
\[
  |p-q(K)|\leq r(K) + C(\alpha)\eta D(X,N)\,.
\]
Choosing $\eta<\delta/(2C(\alpha))$, we have by (c) that
\[|p-q(K)| \leq r(K) - \delta D(X,N)\,,
\]
which in turn implies $|p-\tilde q| < \tilde r$.
This is a contradiction to $\tilde q \in V_x \cap V_y \cap V_z$, since $p$ is closer to $\tilde q$ than any of $x,y,z$. This shows $K_\pi\in \cT'$.

Now we show the implication $K_\pi\in \cT'\Rightarrow K\in \cT$: Let $x,y,z\in X\cap B(x_0,s/2)$ with $[\pi(x),\pi(y),\pi(z)]\in \cT'$. Let $p\in X\cap B(x_0,s) \setminus \{x,y,z\}$. Assume for a contradiction that $|p-\tilde q| \leq \tilde r$. Then again by Lemma \ref{lma: circumcenter regularity} we have
\[
|p - \tilde q| < \tilde r \Rightarrow |p-q| < r + \delta D(X,N) \Rightarrow |p-q| \leq r - \delta D(X,N) \Rightarrow |\pi(p) - q'| < r'.
\]
Here again we used (c) and the fact that $D(X,N)$ is small enough. The last inequality is a contradiction, completing the proof of \eqref{eq:22},  and hence the proof of the present proposition.
\end{proof}

\begin{rem}
A much shorter proof exists for the case of the two-sphere, $N = \mathcal{S}^2$. Here, any finite set $X\subset \mathcal{S}^2$ such that no four points of $X$ are coplanar and every open hemisphere contains a point of $X$ admits a Delaunay triangulation homeomorphic to $\mathcal{S}^2$, namely $\partial \conv(X)$.

Because no four points are coplanar, every face of $\partial \conv(X)$ is a regular triangle $K = [x,y,z]$. The circumcircle of $K$ then lies on $\mathcal{S}^2$ and $q(K) = n(K)|q(K)|$, where $n(K)\in \mathcal{S}^2$ is the outer normal. (The case $q(K)=-|q(K)|n(K)$ is forbidden because the hemisphere $\{x\in \mathcal{S}^2\,:\,x \cdot n(K)>0\}$ contains a point in $X$.) To see that the circumball contains no other point $p\in X\setminus \{x,y,z\}$, we note that since $K\subset \partial \conv(X)$ we have $(p-x)\cdot n(K)< 0$, and thus $|p-q(K)|^2 = 1 + 1 - 2p \cdot q(K) > 1 + 1 - 2x \cdot q(K) = |x-q(K)|^2$.

Finally, $\partial \conv(X)$ is homeomorphic to $\mathcal{S}^2$ since $\conv(X)$ contains a regular tetrahedron.
\end{rem}

We are now in a position to prove the upper bound of our main theorem, Theorem \ref{thm:main} (ii).

\begin{figure}
\includegraphics[height=5cm]{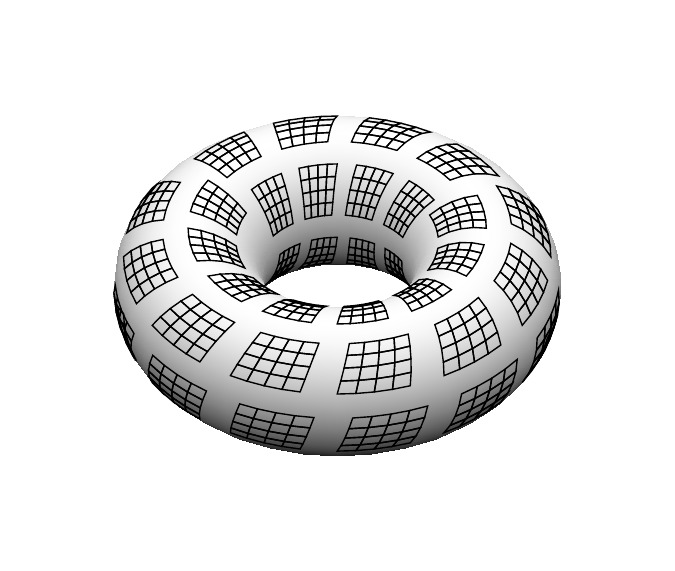}
\caption{The global triangulation of a smooth surface is achieved by first covering a significant portion of the surface with the locally optimal triangulation, then adding additional points in between the regions, and finally finding a global Delaunay triangulation. \label{fig:upper bound}} 
\end{figure}

\begin{proof}[Proof of Theorem \ref{thm:main} (ii)]
We first note that it suffices to show the result for $h\in C^3(M)$ with  $\|h\|_\infty < \frac{\delta(M)}{2}$. To see this, we approximate in the general case $h\in W^{2,2}(M)\cap W^{1,\infty}(M)$, $\|h\|_\infty \leq \frac{\delta(M)}{2}$ by smooth functions $h_{\beta} :=  H_\beta h$, where $(H_\beta)_{\beta >0 }$ is the heat semigroup. Clearly $H_\beta h \in C^\infty(M)$, and $\nabla H_\beta h \to \nabla h$ uniformly, so that $\|h\|_{\infty}\leq \frac{\delta}{2}$ and $\|\nabla h_{\beta}\|_\infty <\|\nabla h\|_{\infty}+1$ for $\beta$ small enough. % We pick a diagonal sequence $\eta(\beta) \to 0$, such that $h_\beta := h_{\eta(\beta),\beta}$ is admissible in the sense that $h_\beta \to h$ in $W^{2,2}(M)$, $\|nabla h_\beta\|_\infty < Z$, $\|h_\beta\|_\infty < \frac{\delta(M)}{2}$. 
% Let $f: \R^2\times \R^{2\times 2}\to \R$ be defined by
% \[
%   f(\xi,A)=\left|(1+\xi\otimes\xi)^{-1}A|^2 (1+|\xi|^2)^{-1/2}\,,
%   \]
%   i.e., the integrand  satisfying
%   \[
%     \int_{M_h}|Dn_{M_h}|\d\H^2=\int

Then 
\[
\int_M f(x,h_\beta(x),\nabla h_\beta(x), \nabla^2 h_\beta) \,\d\Hm^2 \to \int_M f(x,h(x),\nabla h(x), \nabla^2 h) \,\d\Hm^2
\]
for $\beta\to 0$ whenever
\[f:M \times [-\delta(M)/2, \delta(M)/2] \times  B(0,\|\nabla h\|_{\infty}+1)  \times (TM \times TM) \to \R
\]
is continuous with quadratic growth in $\nabla^2 h$.  The Willmore functional
\[
  h\mapsto \int_{M_h} |Dn_{M_h}|^2\d\Hm^2\,,
\]
which is our limit functional, may be written in this way. This proves our claim that we may reduce our argument to the case $h\in C^3(M)$, since the above approximation allows for the construction of suitable diagonal sequences in the strong $W^{1,p}$ topology, for every $p<\infty$.

\medskip

For the rest of the proof we fix $h\in C^3(M)$. We  choose a parameter $\delta>0$. By compactness of $M_h$, there is a finite family of pairwise disjoint closed
open sets  $(Z_i)_{i\in I}$ such that
%balls $(B_i)_{i\in I} = (B(x_i,r_i))_{i\in I}$ with $x_i\in M_h$ such that
\[
\Hm^2\left(M_h \setminus \bigcup_{i\in I} Z_i\right) \leq \delta
\]
and such that, after applying a rigid motion $R_i:\R^3\to \R^3$, the surface $R_i(M_h \cap Z_i)$ is the graph of a function $h_i\in C^2(U_i)$ for some open sets $(U_i)_{i\in I}$ with $\|\nabla h_i\|_\infty \leq \delta$ and $\|\nabla^2 h_i - \diag(\alpha_i,\beta_i)\|_\infty \leq \delta$.

\medskip

We can apply Proposition \ref{prop: local triangulation} to $R_i(M_h \cap Z_i)$ with global parameters $\theta := \delta$ and $\eps>0$ such that $\dist(Z_i,Z_j)>2\e$ for $i\neq j$, yielding point sets $X_{i,\eps}\subset M_h \cap B_i$. The associated triangulated surfaces $\cT_{i,\e}$ (see \ref{fig:upper bound}) have the Delaunay property, have vertices $X_{i,\eps}$ and maximum circumball radius at most $\eps$. Furthermore, we have that
\begin{equation}\label{eq: sum local interactions}
\begin{aligned}
  \sum_{i\in I} &\sum_{K,L\in \cT_{i,\eps}} \frac{l_{KL}}{d_{KL}} |n(K) - n(L)|^2\\
  & \leq (1+C(\delta+\e)) \sum_{i\in I} \L^2(U_i)\times\\
  &\quad \times \left(\max_{x\in U_i}|\partial_{11}h_i(x)|^2+\max_{x\in U_i}|\partial_{22}h_i(x)|^2+\delta^{-1}\max_{x\in U_i}|\partial_{12}h_i(x)|^2\right)+C\e\\
% &\leq (1+C\delta) \sum_{i\in I} \Hm^2(M_h \cap Z_i) \left(\max_{x\in U_i}|\partial_{11}h(x)|^2+\max_{x\in U_i}|\partia_{22}h(x)|^2+\max_{x\in U_i}|\partia_{12}h(x)|^2\\
%   \max_{M_h \cap B_i}|Dn|^2\\
 &\leq (1+C(\delta+\e)) \sum_{i\in I} \int_{M_h \cap Z_i} |Dn_{M_h}|^2 \,\d\Hm^2+C(\e+\delta)\,,
\end{aligned}
\end{equation}
where in the last line we have used  $\|\nabla h_i\|_{\infty}\leq \delta$, $\|\dist(\nabla^2h_i,\diag(\alpha_i,\beta_i)\|_{\infty}\leq \delta$, and the identity
\[
  \begin{split}
\int_{M_h \cap Z_i} |Dn_{M_h}|^2 \,\d\Hm^2&=
\int_{(U_i)_{h_i}}|Dn_{(U_i)_{h_i}}|^2\d\H^2\\
&=\int_{U_i}\left|(\mathbf{1}_{2\times 2}+\nabla h_i\otimes \nabla h_i)^{-1}\nabla^2 h_i\right|^2(1+|\nabla h_i|^2)^{-1/2}\d x\,.
\end{split}
  \]

We shall use the point set $Y_{0,\eps} := \bigcup_{i\in I} X_{i,\eps}$ as a basis for a global triangulated surface. We shall successively augment the set by a single point $Y_{n+1,\eps} := Y_{n,\eps} \cup \{p_{n,\eps}\}$ until the construction below terminates after finitely many steps. We claim that we can choose the points $p_{n,\eps}$ in such a way that for every $n\in\N$ we have
\begin{itemize}
\item [(a)] $\min_{x,y\in Y_{n,\eps}, x\neq y} |x-y| \geq \frac{\eps}{2}$.
\item [(b)] Whenever $x,y,z,p\in Y_{n,\eps}$ are four distinct points such that the circumball $\overline{B(q,r)}$ of $[x,y,z]$ exists and has $r\leq \eps$, then
\[
\left| |p-q| - r \right| \geq \frac{\delta}{2} \eps.
\]
If at least one of the four points $x,y,z,p$ is not in $Y_{0,\eps}$, then
  \begin{equation}\label{eq:21}
\left| |p-q| - r \right| \geq c \eps,
\end{equation}
where $c>0$ is a universal constant.
\end{itemize}

First, we note that both (a) and (b) are true for $Y_{0,\eps}$.

Now, assume we have constructed $Y_{n,\eps}$. If there exists a point $x\in M_h$ such that $B(x,\eps) \cap Y_{n,\eps} = \emptyset$, we consider the set $A_{n,\eps}\subset M_h \cap B(x,\frac{\eps}{2})$ consisting of all points $p\in M_h \cap B(x,\frac{\eps}{2})$ such that for all regular triangles $[x,y,z]$ with $x,y,z\in Y_{n,\eps}$ and circumball $\overline{B(q,r)}$ satisfying $r\leq 2\eps$, we have $\left||p-q| - r\right| \geq c \eps$.

Seeing as how $Y_{n,\eps}$ satisfies (a), the set $A_{n, \eps}$ is nonempty if $c>0$ is chosen small enough, since for all triangles $[x,y,z]$ as above we have
\[
\Hm^2\left(\left\{ p\in B(x,\frac{\eps}{2})\cap M_h\,:\,\left||p-q| - r\right| < c \eps \right\}\right) \leq 4c \eps^2,
\]
and the total number of regular triangles $[x,y,z]$ with $r\leq 2\eps$ and $\overline{B(q,r)}\cap B(x,\eps)\neq \emptyset$ is universally bounded as long as $Y_{n,\eps}$ satisfies (a).

We simply pick $p_{n,\eps}\in A_{n,\eps}$, then clearly $Y_{n+1,\eps} \coloneqq Y_{n,\eps} \cup \{p_{n,\eps}\}$ satisfies (a) by the triangle inequality. We now have to show that $Y_{n+1,\eps}$ still satisfies (b).

This is obvious whenever $p = p_{n,\eps}$ by the definition of $A_{n,\eps}$. If $p_{n,\eps}$ is none of the points $x,y,z,p$, then (b) is inherited from $Y_{n,\eps}$. It remains to consider the case $p_{n,\eps} = x$. Then $x$ has distance $c\eps$ to all circumspheres of nearby triples with radius at most $2\eps$. We now assume that the circumball $\overline{B(q,r)}$ of $[x,y,z]$ has radius $r \leq \eps$ and that some point $p\in Y_{n,\eps}$ is close to $\partial B(q,r)$. To this end, define
\[
\eta \coloneqq \frac{\left||p-q| - r \right|}{\eps}\,.
\]

We show that $\eta \geq \eta_0$ for some universal constant. To this end, we set

\[
  p_t \coloneqq (1-t)p + t\left(q+r\frac{p-q}{|p-q|}\right)
\]
(see  Figure \ref{fig:pt}) and note that if $\eta\leq \eta_0$, all triangles $[y,z,p_t]$ are uniformly regular.

\begin{figure}[h]
  \centering
  \includegraphics[height=5cm]{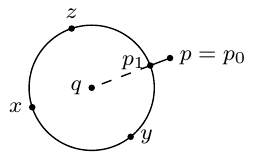}
  \caption{The definition of $p_t$ as linear interpolation between $p_0$ and $p_1$. \label{fig:pt}}
\end{figure}

Define the circumcenters $q_t \coloneqq q(y,z,p_t)$, and note that $q_1 = q$. By Lemma \ref{lma: circumcenter regularity}, we have $|q_1 - q_0| \leq C|p_1 - p_0| = C\eta \eps$ if $\eta\leq \eta_0$. Thus the circumradius of $[y,z,p_0]$ is bounded by
\[
  |y-q_0| \leq |y-q| + |q-q_0| \leq (1+C\eta)\eps \leq 2\eps
\]
if $\eta\leq \eta_0$. Because $x\in Y_{n+1,\eps} \setminus Y_{n,\eps} \subset A_{n,\eps}$, we have, using \eqref{eq:21},
\[
c\eps \leq \left| |x-q_0| - |p-q_0|\right| \leq \left| |x-q| - |p-q| \right| + 2 |q - q_0| \leq (1+2C)\eta\eps,
\]
i.e. that $\eta \geq \frac{c}{1+2C}$. This shows (b).

\begin{comment}
 We note that by (a) we have $r\geq \frac{\eps}{4}$. We set $p_t \coloneqq (1-t) p + t\left((q + r\frac{p-q}{|p-q|}\right)$ for $t\in[0,1]$. If $\eta<\eta_0$, then the triangles $[y,z,p_t]$ are all $\zeta_0$-regular  triangles for some universal constants $\zeta_0,\eta_0>0$. By Lemma \ref{lma: circumcenter regularity}, then $|q(y,z,p_0) - q(y,z,p)|\leq C \eta \eps$. However, $q(y,z,p_0) = q$, and $|p-q(y,z,p)| \leq 2\eps$ for $\eta<\eta_0$. By the choice $x\in A_{n,\eps}$ then
\[
\left| |x-q(y,z,p)| - |p-q(y,z,p)| \right| \geq c\eps,
\]
which implies that
\[
\left||x-q| - |p-q| \right| \geq c \eps - C \eta \eps,
\]
i.e. that $\eta \geq \min\left( \frac{1}{1+C}, \eta_0, \frac14 \right)$, which is a universal constant. This shows (b).
\end{comment}

Since $M_h$ is compact, this construction eventually terminates, resulting in a set $X_\eps := Y_{N(\eps),\eps} \subset M_h$ with the properties (a), (b), and $D(X_\eps,M) \leq \eps$.

\medskip

Consider a Lipschitz function  $g:M_h\to \R$. Since  $M_h$ is a $C^2$ surface, we have that for $\|g\|_{W^{1,\infty}}$ small enough, $(M_h)_g$ is locally a tangent Lipschitz graph over $M$, see Definition \ref{def:Mgraph} (iii). By Lemma \ref{lma: graph property}, this implies that $(M_h)_g$ is a graph over $M$.

Invoking Proposition \ref{prop: Delaunay existence} yields a Delaunay triangulated surface $\cT_\eps \coloneqq \cT(X_\eps, M_h)$ with vertex set $X_\eps$ that is $\zeta_0$-regular for some $\zeta_0>0$, and $\bigcup_{K\in \cT_\eps} = (M_h)_{g_\eps}$ with $\|g_\eps\|_{W^{1,\infty}}\leq C(\delta)\eps$.

By the above, there exist Lipschitz functions $h_\e:M\to \R$ such that
$(M_h)_{g_\eps} = M_{h_\eps}$, with $h_\eps \to h$ in $W^{1,\infty}$, $\|h_\eps\|_\infty \leq \frac{\delta(M)}{2}$ and $\|\nabla h_\e\|\leq \|\nabla h\|_{\infty}+1$.

\medskip

It remains to estimate the energy. To do so, we look at the two types of interfaces appearing in the sum
\[
 \sum_{K,L\in \cT_\eps} \frac{l_{KL}}{d_{KL}} |n(K) - n(L)|^2.
\]

First, we look at pairwise interactions where $K,L\in \cT(X_{i,\eps})$ for some $i$. These are bounded by \eqref{eq: sum local interactions}.

Next, we note that if $\eps < \min_{i\neq j \in I} \dist(B_i,B_j)$, it is impossible for $X_{i,\eps}$ and $X_{j,\eps}$, $i\neq j$, to interact.

Finally, we consider all interactions of neighboring triangles $K,L\in \cT_\eps$ where at least one vertex is not in $Y_{0,\eps}$. By \eqref{eq:21}, these pairs all satisfy $\frac{l_{KL}}{d_{KL}} \leq C$ for some universal constant $C$ independent of $\eps,\delta$, and $|n(K) - n(L)|\leq C\eps$ because $\cT$ is $\zeta_0$-regular and $M_h$ is $C^2$. Further, no points were added inside any $B_I$. Thus
\[
  \begin{split}
  \sum_{\substack{K,L\in \cT_\eps\,:\,\text{at least}\\\text{ one vertex is not in }Y_{0,\eps}}}& \frac{l_{KL}}{d_{KL}} |n(K) - n(L)|^2 \\
  &\leq C\Hm^2\left(M_h \setminus \bigcup_{i\in I}B(x_i, r_i - 2\eps)\right)\\
  &\leq C \delta + C(\delta)\eps.
\end{split}
\]

Choosing an appropriate diagonal sequence $\delta(\eps) \to 0$ yields a sequence $\cT_\eps = M_{h_\eps}$ with $h_\e\to h$ in  $W^{1,\infty}(M)$ with
\[
\limsup_{\eps \to 0} \sum_{K,L\in \cT_\eps} \frac{l_{KL}}{d_{KL}} |n(K) -n(L)|^2 \leq \int_{M_h} |Dn_{M_h}|^2\,d\Hm^2.
\]
\end{proof}

\section{Necessity of the Delaunay property}
\label{sec:necess-dela-prop}

We now show that without the Delaunay condition, it is possible to achieve a lower energy. In contrast to the preceding sections, we are going to choose an underlying manifold $M$ with boundary (the ``hollow cylinder'' $S^1\times[-1,1]$). By ``capping off'' the hollow cylinder one can construct a counterexample to the lower bound in Theorem \ref{thm:main}, where it is assumed that  $M$ is  compact without boundary.

\begin{prop}\label{prop: optimal grid}
Let $M =S^1\times[-1,1] \subset \R^3$ be a
hollow cylinder and $\zeta>0$. Then there are $\zeta$-regular triangulated
surfaces $\cT_j\subset \R^3$ with $\size(\cT_j) \to 0$ and $\cT_j \to M$ for
$j\to\infty$ with
\[
\limsup_{j\to\infty} \sum_{K,L\in \cT_j} \frac{\l{K}{L}}{d_{KL}}
|n(K)-n(L)|^2 < c(\zeta) \int_M |Dn_M|^2\,d\H^2\,,
\]
where the positive constant $c(\zeta)$ satisfies
\[
  c(\zeta)\to 0 \quad \text{ for } \quad\zeta\to 0\,.
  \]
\end{prop}

\begin{figure}
\includegraphics[height=5cm]{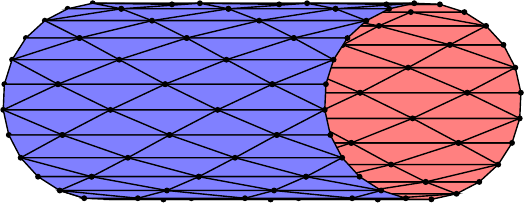}
\caption{A non-Delaunay triangulated cylinder achieving a low energy . \label{fig:cylinder}} 
\end{figure}

\begin{proof}
For every $\eps = 2^{-j}$ and $s\in\{2\pi j^{-1}:j=3,4,5,\dots\}$, we
define a flat triangulated surface $\cT_j\subset \R^2$ with
$\size(\cT_j) \leq \e$ as follows: As manifolds with boundary, $\cT_j=[0,2\pi]\times
[-1,1]$ for all $j$; all  triangles are isosceles,
with one side a translation of $[0,\eps]e_2$ and height $s\eps$ in
$e_1$-direction. We neglect the triangles close to the boundary
$[0,2\pi]\times\{\pm 1\}$, and leave it to
the reader to verify that their contribution will be negligeable in the end. 

\medskip

We then wrap this triangulated surface around the cylinder, mapping the
corners of triangles onto the surface of the cylinder via  $(\theta,t) \mapsto
(\cos\theta, \sin\theta, t)$, to obtain a triangulated surface $\tilde \cT_j$.
Obviously, the topology of $\tilde \cT_j$ is $S^1\times[-1,1]$. % The resulting triangulated grid is shifted to align with the point $(0,-1)$ and augmented by the isosceles boundary triangles with remaining side a translation of $[0,s\eps]e_2$ and then rolled around the cylinder.

Then we may estimate all terms $\frac{\l{K}{L}}{d_{KL}} |n(K) - n(L)|^2$. We
first find the normal of the reference triangle $K\in \tilde \cT_j$ spanned by the points $x = (1,0,0)$, $y = (1,0,\eps)$, and $z = (\cos(s\eps),\sin(s\eps),\eps/2)$. We note that
\[
n(K) = \frac{(y-x) \times (z-x)}{|(y-x) \times (z-x)|} = \frac{(-s\eps\sin(s\eps), s\eps(\cos(s\eps)-1),0)}{s\eps(2-2\cos(s\eps))} = (1,0,0) + O(s\eps).
\]

We note that the normal is the same for all translations $K+te_3$ and for all triangles bordering $K$ diagonally. The horizontal neighbor $L$ also has $n(L) = (1,0,0) + O(s\eps)$. However, we note that the dimensionless prefactor satisfies $\frac{\l{K}{L}}{d_{KL}} \leq \frac{2\eps}{\eps/s} = s$. Summing up the $O(s^{-1}\eps^{-2})$ contributions yields
\[
\sum_{K,L\in \cT_j} \frac{\l{K}{L}}{d_{KL}} |n(K) - n(L)|^2 \leq C \frac{s^3\eps^2}{s\eps^2} = Cs^2.
\]

This holds provided that $\eps$ is small enough. Letting $s\to 0$, we see that this energy is arbitrarily small.
% Finally, we augment the triangulation $\cT_j$ to cover the two hemispheres. Here it is possible by the upper bound construction to achieve an energy as low as $8\pi$, which is $\int_{S^2} |Dn|^2\,d\H^2$.
\end{proof}

\bibliographystyle{alpha}
\bibliography{triangulations}
\end{document}